\documentclass[a4paper,10pt]{amsart}

\usepackage{tikz}
\usepackage{pgfplots}
\usepackage{cleveref}
\pgfplotsset{compat=newest}
\usetikzlibrary{plotmarks}

\usepackage{graphicx}
\usepackage{color}
\usepackage{xcolor}
\usepackage{amsmath,mathtools}
\usepackage{amsfonts}
\usepackage{mathrsfs}
\usepackage{color}
\usepackage{enumerate}
\usepackage{algorithm}
\usepackage{mdframed}
\usepackage{algorithmic}
\usepackage{siunitx}
\usepackage{multirow}
\usepackage{stackrel}
\usepackage{booktabs}
\usepackage{tabularx}
\usepackage{arydshln}
\usepackage{cite}
\usepackage{url}

\definecolor{darkgreen}{rgb}{.39 0.84 0.10}

\newcommand{\bGamma}{\boldsymbol{\Gamma}}
\newcommand{\bgamma}{\boldsymbol{\gamma}}

\usepackage{changes}
\definechangesauthor[color=orange]{AP}

\newcommand{\tLB}{\emph{LB}}
\newcommand{\tUB}{\emph{UB}}
\newcommand{\I}{\mathcal{I}}

\newtheorem{Theorem}{Theorem}

\newtheorem{Corollary}{Corollary}
\newtheorem{Definition}{Definition}

\newtheorem{Lemma}{Lemma}

\newtheorem{Remark}{Remark}

\newcommand{\R}{\mathbb{R}}
\renewcommand{\P}{\mathcal{P}}
\usepackage[top=3.0cm,bottom=3.0cm,left=2.2cm,right=2.2cm]{geometry}
\usepackage{lineno}

\begin{document}
\title{Eigenvalue bounds for preconditioned symmetric multiple saddle-point matrices}
\author{Luca Bergamaschi  \and
\'Angeles Mart\'inez \and John W. Pearson \and Andreas Potschka}
\footnotetext[1]{Department of Civil Environmental and Architectural Engineering, University of Padua, Via Marzolo, 9, 35100 Padua, Italy,
E-mail: \texttt{luca.bergamaschi@unipd.it}}
\footnotetext[2]{Department of Mathematics, Informatics and Geosciences, University of Trieste, Via Alfonso Valerio, 12/1, 34127 Trieste, Italy,
E-mail: \texttt{amartinez@units.it}}
\footnotetext[3]{School of Mathematics, The University of Edinburgh, James Clerk Maxwell Building, The King's Buildings, Peter Guthrie Tait Road, Edinburgh, EH9 3FD, United Kingdom,
E-mail: \texttt{j.pearson@ed.ac.uk}}
\footnotetext[4]{Institute of Mathematics, Clausthal University of Technology, Erzstr.~1, 38678 Clausthal-Zellerfeld, Germany,
E-mail: \texttt{andreas.potschka@tu-clausthal.de}}

\begin{abstract}
We develop eigenvalue bounds for symmetric, block tridiagonal multiple saddle-point linear systems, preconditioned with block diagonal matrices.
	We extend known results for $3 \times 3$ block systems {[Bradley \& Greif, IMA J.\ Numer.\ Anal.\ 43 (2023)]}
	and for $4 \times 4$ systems [Pearson \& Potschka, IMA J.\ Numer.\ Anal.\ 44 (2024)]
	to an arbitrary number of blocks. Moreover, our results generalize the bounds in {[Sogn \& Zulehner, IMA J.\ Numer.\ Anal.\ 39 (2018)]}, developed for an arbitrary number of blocks with null diagonal blocks.  Extension to the bounds when the Schur complements are approximated is also provided, using perturbation arguments. Practical bounds are also obtained for the double saddle-point linear system. Numerical experiments validate our findings.
\end{abstract}
\maketitle

\bigskip


\section{Introduction}
We consider the iterative solution of a block tridiagonal multiple saddle-point linear system $\mathcal{A} w = f$, where
\begin{equation*}
	\mathcal{A} = \begin{bmatrix}
	A_0 & B_1^\top & 0 & \cdots & 0  \\
	B_1 & -A_1& B_2^\top & \ddots & \vdots \\
	0 &  B_2 &A_2  & \ddots & 0 \\
	\vdots	& \ddots & \ddots & \ddots & B_N^\top \\[.3em]
	0	& \cdots & 0 & B_N & (-1)^N A_N
\end{bmatrix}.
\end{equation*}
We assume that $A_0 \in \R^{n_0 \times n_0}$ is symmetric positive definite, all other square block matrices $A_k \in \mathcal{M}_{n_k \times n_k}$ are symmetric positive semi-definite
and $B_k \in \R^{n_k \times n_{k-1}}$ have full rank (for $k = 1, \ldots, N$).
We further assume that ${n_{k} \le n_{k-1}}$ (for $k = 1, \ldots, N$).
These conditions are sufficient (though not necessary) to ensure invertibility of $\mathcal{A}$. Less restrictive conditions,
in the case of double saddle-point matrices, have been discussed, e.g., in \cite{BEIK2024403}.

In this work, we develop eigenvalue bounds for the symmetric multiple saddle-point matrix {$\mathcal{A}$}, preconditioned with block diagonal Schur complement preconditioners.
Linear systems involving matrix {$\mathcal{A}$} arise, for example with $N =2$ (double saddle-point systems), in many scientific applications including constrained least squares problems \cite{Yuan}, constrained quadratic programming \cite{Han}, magma--mantle dynamics \cite{Rhebergen}, 
liquid crystal director modeling \cite{RamGar2023} or in the coupled Stokes--Darcy problem, and the preconditioning of such linear systems has been considered in, e.g., \cite{Balani-et-al-2023a, Balani-et-al-2023b, Szyld, Benzi2018, BeikBenzi2022}.
In particular, block diagonal preconditioners for  matrix {$\mathcal{A}$} have been thoroughly studied in
\cite{Bradley,PPNLAA24}.

More generally, multiple saddle-point linear systems with $N > 2$ have been addressed in \cite{SZ} where the block diagonal matrix with positive definite Schur complements blocks is considered as the preconditioner, and bounds on the condition number
of the preconditioned matrix have been developed. Furthermore, the eigenvalues of the preconditioned matrix
are exactly computed, in the simplified case where $A_k$ are $n_k \times n_k$ zero matrices for $k > 0$, and are related to the zeros
of suitable Chebyshev-like polynomials. A new symmetric positive definite preconditioner
has recently been proposed in \cite{pearson2023symmetric} and successfully tried on constrained optimization problems up to $N=4$. The spectral properties of
the resulting preconditioned matrix have been described in \cite{BMPP_COAP24} for $N=2$. Preconditioning of the $4 \times 4$ saddle-point linear system ($N=3$), although
with a different block structure, has been addressed in  \cite{BSZ2020} to solve a class of optimal control problems.
A practical preconditioning strategy for multiple saddle-point linear systems, based on sparse approximate inverses of the diagonal blocks of the 
block diagonal Schur complement preconditioning matrix, is proposed in \cite{FerFraJanCasTch19}, for the solution of 
coupled poromechanical models and the mechanics of fractured media.	

In this work, we extend the results in \cite{Bradley} for $3 \times 3$ block systems and those for $4 \times 4$ block systems given in \cite{pearson2023symmetric} to an arbitrary number of blocks. Moreover, our results generalize the bounds in \cite{SZ}, developed for arbitrary $N$, but with null diagonal blocks. We are also interested in settings where the (inverses of the) Schur complements are applied approximately, so we extend the eigenvalue bounds to such settings using generic perturbation arguments. In the double saddle-point setting, we may additionally apply the methodology of this work to devise bespoke bounds which reflect the quality of the approximations of the respective Schur complements. Altogether, this paper provides a range of strategies to quantify the spectra of preconditioned systems, with exact and inexact applications of the (inverse) Schur complements.

This paper is structured as follows: We state the block diagonal Schur complement preconditioner in Section \ref{sec:bd_sc_prec}, and relate the eigenvalues of the preconditioned system matrix to the zeros of a family of polynomials defined via a three-term recurrence. We derive bounds for their extremal zeros and, thus, also for the extremal eigenvalues in Section \ref{sec:polynomials}. We provide two approaches for the analysis of approximate versions of the preconditioner in Section \ref{sec:inexact}, one general result based on a backward analysis argument and one refined result for double saddle-point systems. 
For double saddle-point linear systems with approximated Schur complements, the bounds are verified using synthetic test cases, as well as realistic problems arising from PDE-constrained optimization. Some concluding remarks are provided in Section \ref{sec:conc}.

\section{Characterization of the eigenvalues of the block diagonally preconditioned matrix}
\label{sec:bd_sc_prec}
We consider preconditioning a linear system involving $\mathcal A$ as the coefficient matrix, with the block diagonal preconditioner defined as:
\begin{equation*}
	\mathcal{P}_D = \text{blkdiag} \left(S_0, S_1,  \ldots, S_N\right),
\end{equation*}
\[ \text{with} \qquad S_0 = A_0, \qquad S_k = A_k + B_k S_{k-1}^{-1} B_k^\top, \quad k = 1, \ldots, N.\]
The previously mentioned conditions on the rectangular matrices $B_k$ ensure that all Schur complements $S_0, S_1, \ldots, S_N$
are symmetric positive definite and thus invertible. In the following, to simplify the notation, we will denote as $I$ the identity matrix, the dimension of which will be clear
from the context.

To investigate the eigenvalues of $\mathcal{P}_D ^{-1} \mathcal{A}$ we consider instead the similar symmetric matrix:
\begin{align}
	\label{compact}
	&\mathcal{Q} := \mathcal{P}_D ^{-1/2} \mathcal{A} \mathcal{P}_D ^{-1/2}  \\
\nonumber 	&=	\begin{bmatrix}
	I & R_1^\top & 0 & \cdots & 0 \\
	R_1 & -E_1 & R_2^\top & \ddots & \vdots \\
	0 &  R_2 &E_2  & \ddots & 0 \\
	\vdots	& \ddots & \ddots & \ddots & R_N^\top \\[.3em]
	0	& \cdots & 0 & R_N & (-1)^N E_N
	\end{bmatrix}
=
	\label{Ecompact}
	\begin{bmatrix}
	I & R_1^\top & 0 & \dots & 0 \\
	R_1 & -I + R_1 R_1^\top & R_2^\top & \ddots & \vdots \\
	0 &  R_2 &I -R_2R_2^\top   & \ddots & 0 \\
	\vdots	& \ddots & \ddots & \ddots & R_N^\top \\[.3em]
	0	& \cdots & 0 & R_N & (-1)^N (I - R_N R_N^\top )
	\end{bmatrix},
\end{align}
where
\[ R_k = S_{k}^{-1/2} B_k S_{k-1} ^{-1/2}, \qquad E_k =  S_{k}^{-1/2}  A_k  S_{k}^{-1/2}, \]
which implies that
\begin{equation} \label{RRE} 
R_k R_k^\top + E_k = I.
\end{equation}
Componentwise, we write the eigenvalue problem $\mathcal Q u = \lambda u$, with $u = \begin{bmatrix} u_1^\top, \ldots, u_{N+1}^\top\end{bmatrix}^\top$, as
\begin{equation}
	\begin{array}{lclclclcl}
		 		u_1 &{}+& R_1^\top u_2  &{}&&&  & = & \lambda u_1, \\
		 		R_1 u_1 &{}+& (R_1 R_1^\top- I) u_2  &{}+& R_2^\top u_3  && & = & \lambda u_2, \\
		 		&& R_2 u_2 &{}+& (I - R_2 R_2^\top) u_3  &{}+& R_3^\top u_4  & = & \lambda u_3, \\
		&{}&& \vdots &{}& \vdots && \vdots &  \\
				&{}& R_{N-1} u_{N-1} &{}+& (-1)^{N-1}(R_{N-1} R_{N-1}^\top - I) u_{N}  &+& R_N^\top u_{N+1}  & = & \lambda u_{N}, \\
		 		&{}&& & R_{N} u_{N} &{}+& (-1)^N(R_{N} R_{N}^\top- I) u_{N+1} & = & \lambda u_{N+1}.
	\end{array}
\label{dpsp}
\end{equation}
The matrices $R_k R_k^\top$ are all symmetric positive definite. We define $\gamma_k$ as a generic value of the
Rayleigh quotient of $R_k R_k^\top$, that is
\[ \gamma_k \equiv \gamma_k(w) = \frac{w^\top R_k R_k^\top w}{w^\top w}, \qquad \text{with} \quad 0 < \gamma_k \le 1, \]
the upper bound arising from the fact that $R_k R_k^\top \preceq I$, due to \eqref{RRE}.
Moreover, we define the vector
\begin{equation*}
	\boldsymbol{\gamma} \equiv  \begin{bmatrix} \gamma_1, \ldots, \gamma_N\end{bmatrix}.
		\end{equation*}
In the next section, we will characterize the eigenvalues of the preconditioned matrix $\mathcal Q$ in terms of the zeros of the following sequence of
parametric polynomials:
\begin{Definition}
\begin{align}
	\nonumber U_0(x, \boldsymbol{\gamma}) & =  1, \\
	\nonumber U_1(x, \boldsymbol{\gamma}) & =  x - 1, \\
	\nonumber U_2(x, \boldsymbol{\gamma}) & =  x^2 - \gamma_1 x - 1, \\
\label{un} U_{k+1}(x, \boldsymbol{\gamma})  &=  ((-1)^{k+1}(1-\gamma_k) + x) U_k(x, \boldsymbol{\gamma}) - \gamma_k U_{k-1}(x, \boldsymbol{\gamma}), \quad k \ge 1. 
\end{align}
\end{Definition}

\bigskip

\noindent
	\textbf{Notation}. We will denote as 
$\mathcal{I}(Q)$ the union of the intervals 
containing the zeros of a $\bgamma$-dependent polynomial $Q(x, \bgamma)$ for $\bgamma \in [0, 1]^N$. We use the simplified notation
$\mathcal{I}_k = \mathcal{I}(U_k(\lambda, \boldsymbol{\gamma}))$ to {denote bounds for the roots of polynomials of the form $U_k$} over the
valid range of $\boldsymbol\gamma$.
{With $\sigma(A)$ we denote the spectrum of the square matrix $A$.}

\bigskip

We now define
\[\boldsymbol{\Gamma} = \begin{bmatrix} R_1 R_1^\top, R_2 R_2^\top, \ldots, R_N R_N^\top\end{bmatrix}
	,\]  and a sequence of symmetric
	matrix-valued  functions $\{Z_k\}_{k = 1, \ldots, N+1} \in \R^{n_{k-1} \times n_{k-1}}$.
\begin{Definition}
\begin{align*}
	Z_1(\lambda, \boldsymbol{\Gamma}) & =  (\lambda  - 1)I, \\
	Z_2(\lambda, \boldsymbol{\Gamma}) & =  -\frac{R_1 R_1^\top}{\lambda-1} + I - R_1 R_1^\top + \lambda I, \quad \lambda \ne 1 \\
	Z_{k+1}(\lambda, \boldsymbol{\bGamma})  &=  -R_k Z_{k}(\lambda, \bGamma)^{-1}R_k^\top + \lambda I +  (-1)^{k+1}(I - R_k R_k^\top), \quad k \ge 1 ~ \emph{ and } ~ \lambda ~ \emph{ s.t. } ~ 0 \not \in \sigma\left(Z_{k}(\lambda, \bGamma)\right).
\end{align*}
\end{Definition}
		We premise a technical lemma whose proof can be found in \cite[Lem. 2.1]{BMPP_COAP24}.
\begin{Lemma}\label{Le1}
        Let $Z$ be a symmetric matrix valued function  defined in  $F \subset \R$, and
        \[0 \notin [\min\{\sigma(Z(\zeta))\},\max\{\sigma(Z(\zeta))\}],\quad {\text{for all } \zeta \in F.}\]
        Then, for arbitrary $s \neq 0$, there exists a vector $v \neq 0$ such that
\begin{equation*}
        \frac{s^\top Z(\zeta)^{-1}s}{s^\top s}=\frac{1}{\gamma_Z},
        \qquad \text{with} \quad \gamma_Z = \frac{v^\top Z(\zeta) v}{v^\top v}.
\end{equation*}
\end{Lemma}

		The next lemma will be used in the proof of the subsequent \Cref{theo2}.
\begin{Lemma}
	\label{lem2}
	For every $u \ne 0$, there is a choice of $\boldsymbol{\gamma}$ for which
	\[ \frac{u^\top Z_{k+1}(\lambda, \bGamma) u}{u^\top u}  = \frac{U_{k+1}(\lambda, \boldsymbol{\gamma}) }{U_{k}(\lambda, \boldsymbol{\gamma})} \qquad {\text{for all }} \lambda \not \in \bigcup_{j=1}^{k} \mathcal I_j. \]
\end{Lemma}
\begin{proof}
	This is shown by induction. For $k=0$ we have $\dfrac{u^\top Z_1(\lambda, \boldsymbol{\Gamma}) u}{u^\top u} = \lambda -1 = \dfrac{U_1(\lambda, \bgamma)}{U_0(\lambda, \bgamma)}
	\ \text{ for all } \lambda \in \R$. 
	If $k \ge 1$,  {the condition
	$\lambda \not \in \mathcal I_k$, together with
	the inductive hypothesis 
	$\dfrac{u^\top Z_{k}(\lambda, \bGamma) u}{u^\top u}  = \dfrac{U_{k}(\lambda, \boldsymbol{\gamma}) }{U_{k-1}(\lambda, \boldsymbol{\gamma})}$, implies invertibility of $Z_{k}(\lambda, \bGamma)$. Moreover, {by the resulting definiteness of $Z_k$ this} is equivalent to the condition
	$0 \not \in [\min\{\sigma(Z_{k}(\lambda, \bGamma))\},\max\{\sigma(Z_{k}(\lambda, \bGamma))\}]$,
	which guarantees that \Cref{Le1} can be applied.}
	Therefore, we can write
	\begin{eqnarray}
		\frac{u^\top Z_{k+1}(\lambda, \bGamma) u}{u^\top u}   & = &
		-\frac{u^\top R_k Z_{k}(\lambda, \bGamma)^{-1}R_k^\top u}{u^\top u} + \lambda
		+  (-1)^{k+1}(1 - \gamma_k)\nonumber \\
		&\underbrace{=}_{\text{\normalsize{$w = R_k^\top u$}}} & -\gamma_k \frac{w^\top Z_{k}(\lambda, \bGamma)^{-1} w}{w^\top w} + \lambda +  (-1)^{k+1}(1 - \gamma_k). \label{Zrec}
	\end{eqnarray}
		We then apply \Cref{Le1} and the inductive hypothesis to write
		\[ \frac{w^\top Z_{k}(\lambda, \bGamma)^{-1} w}{w^\top w} = \frac
		{U_{k-1}(\lambda, \boldsymbol{\gamma}) }
		{U_{k}(\lambda, \boldsymbol{\gamma})}. \]
		Substituting into \eqref{Zrec} and using the relation \eqref{un} yields the assertion.
\end{proof}

\begin{Theorem}
	Any eigenvalue of $\mathcal{P}_D^{-1} \mathcal{A}$ is located in $\displaystyle \bigcup _{k=1}^{N+1} \mathcal{I}_k$. 
	\label{theo2}
\end{Theorem}

\begin{proof}
	To show the statement, we define a candidate eigenvalue and prove by induction that for every $k \le N+1$ either
	\[ 	\text{(i)} \ \lambda \in  \mathcal{I}_k \quad \text{or} \quad  \text{(ii)}  \ u_k = Z_k^{-1} R_k^\top u_{k+1}, \]
	and for $k = N+1$ only condition (i) can hold.

	Let $u = \begin{bmatrix} u_1^\top, \ldots, u_{N+1}^\top\end{bmatrix}^\top$ be an eigenvector of $\mathcal Q$, 
		hence satisfying 
		\eqref{dpsp}.
	If $u_2 = 0$ then
$\lambda = 1$ is eigenvalue, clearly
	satisfying $U_1(1, \boldsymbol \gamma) = 0$.
	If $\lambda \ne 1$ we obtain $u_1$ from the first equation of \eqref{dpsp},
	\[ u_1 = \frac{R_1^\top u_2}{\lambda-1 } = Z_1^{-1}R_1^\top u_2,\]
and substitute it into the second one:
\[ \left( -\frac{R_1 R_1^\top}{\lambda-1} + I - R_1 R_1^\top + \lambda I\right)u_2  = R_2^\top u_3,
	\qquad Z_2(\lambda,\boldsymbol{\Gamma}) u_2 = R_2^\top u_3.\]
	Since by \Cref{lem2} we have that
\[ \frac{u_2^\top Z_2(\lambda,\boldsymbol{\Gamma}) u_2}{u_2^\top u_2} =
	\frac{U_2(\lambda, \boldsymbol{\gamma})}{U_1(\lambda, \boldsymbol{\gamma})},
\]
	 if $u_3 = 0$ then $\lambda$ being an eigenvalue implies that $U_2(\lambda, \boldsymbol{\gamma}) = 0$.
	Otherwise, if $\lambda \not \in \mathcal{I}_2$ then $Z_2$ is definite and hence invertible, therefore we may write
	 \[u_2 = Z_2^{-1} R_2^\top u_3, \]
	 Assume now the inductive hypothesis holds for $k-1$. {If $\lambda \not \in \mathcal{I}_{k-1}$, then we can apply Lemma 2,
	 since in  this case we also have that $Z_{k-1}$ is definite and hence invertible.}
	 We can write
	 \[ u_{k-1} = Z_{k-1}^{-1} R_{k-1}^\top u_{k}. \]
	 Substituting this into the $k$th equation of \eqref{dpsp} yields
	 \[ Z_{k} (\lambda, \boldsymbol{\Gamma}) u_{k} \equiv  \left(-R_{k-1} Z_{k-1}(\lambda, \bGamma)^{-1} R_{k-1}^\top + (-1)^{k} (I - R_{k-1} R_{k-1}^\top) + \lambda I\right)u_{k} = R_{k}^\top u_{k+1}.\]
	 Now, if $\lambda$ is an eigenvalue with
$u_{k+1}=0$, from 
	 \[ 0 = \frac{u_k^\top Z_{k}(\lambda, \boldsymbol{\Gamma}) u_{k}}{u_{k}^\top u_{k}} = \frac{U_k(\lambda, \boldsymbol{\gamma})}{U_{k-1}(\lambda, \boldsymbol{\gamma})}
\]
	we have $U_k(\lambda, \boldsymbol{\gamma}) = 0$ and
hence $\lambda \in \mathcal{I}_{k}$.
	Otherwise
	 for any $\lambda \not \in \mathcal{I}_{k}$ we may write 
	 \[ u_{k} = Z_{k}^{-1} R_{k}^\top u_{k+1}. \]
	 The induction process ends for $k = N+1$. In this case, from
	\[ Z_{N+1}  u_{N+1} = 0,\]
	 we have that condition (i) holds, noticing that $u_{N+1}$  cannot be zero, as this would imply that $0 = u_{N+1} = u_{N}
	 = \ldots = u_1$, contradicting the definition of an eigenvector.
\end{proof}



\section{Bounds for the extremal zeros}
\label{sec:polynomials}

Theorem \ref{theo2} guides us that to obtain eigenvalues of our preconditioned system of interest, we should establish appropriate bounds for the intervals $\mathcal{I}_k$, which we accomplish in this section by considering suitable sequences of polynomials of the general form $U_k$.
        
We start by denoting as $\boldsymbol 1_j$ a vector of ones of length $j$,
	and consider two particular cases of the polynomials defined in \eqref{un}: 
\begin{Definition}
	\label{defTV} \,
\begin{enumerate}
	\item[\emph{(i)}] $P_k(x) \equiv U_k(x, \boldsymbol{1}_{k-1})$,
	\item[\emph{(ii)}] $V_k(x) \equiv U_k(x,\begin{bmatrix} 0, \boldsymbol{1}_{k-2}\end{bmatrix})$.
			\end{enumerate}
\end{Definition}
			In case (i), the recurrence relation defines the same polynomials
			considered in \cite[Sec. 2 \& App. A]{SZ} where it is shown that the zeros of $P_k(x)$ are
			\[ p_j = 2 \cos \left (\frac{(2j-1)\pi} {2k+1}\right), \quad j =1, \ldots, k.\]
			For case (ii), we show by induction that
				\begin{equation}\label{Vrecur}
				V_{k+1}(x)  = (-1)^{k} (x-1) P_k(-x), \quad {k \ge 1.}
                \end{equation}
				In fact,
			\begin{align*}
				V_2(x) &= x^2 -1  = -(x-1) P_1(-x)\\
				V_3(x) &= x (x^2-1) - (x-1) = (x-1) (x^2 + x -1) =			(x-1) P_2(-x).
			\end{align*}
			Supposing that \eqref{Vrecur} defines polynomials $V_*$ up to and including $V_k$, then
			\begin{align*} V_{k+1}(x)  &=  x V_k(x) - V_{k-1}(x) =
				(-1)^{k+1} x (x-1) P_{k-1}(-x) - (-1)^k (x-1)P_{k-2}(-x)  \\
				&= (-1)^k (x-1)\left(-x P_{k-1}(-x) - P_{k-2}(-x) \right)
			= (-1)^k (x-1) P_{k}(-x).\end{align*}
			Hence the zeros of $V_{k+1}(x)$ are defined by the set $\{1, v_1, \ldots, v_k\}$ where
			\[ v_j = -2 \cos \left (\frac{(2j-1)\pi} {2k+1}\right), \quad j = 1, \ldots, k.\]

	\begin{Remark}
		In recurrence \eqref{un}, if $\gamma_k = 0$  the roots of $U_{k+1}$ are the roots
		of $U_k$ as well as $1$ or $-1$.
		\label{rem}
	\end{Remark}

We are interested in finding bounds on the zeros of $U_{k+1}(x, \boldsymbol{\gamma}), \ k = 0, \ldots, N-1$.
To this end, we could directly use a recent result in \cite[Lem. 2.2]{BMPP_COAP24}. However, we find it more elegant for our purposes here to use the following result, which can be used as an alternative route to prove~\cite[Lem. 2.2]{BMPP_COAP24}:

\begin{Lemma} \label{lem:how_roots_move}
  Let $d \in \mathbb{N}$ and let $q: \R \times \R^d \to \R$ be continuously differentiable. If $(\lambda^\ast$, $\boldsymbol{\gamma}^\ast) \in \R \times \R^d$ satisfy
  \[
    q(\lambda^\ast, \boldsymbol{\gamma}^\ast) = 0 \quad \text{and} \quad
    \frac{\partial q}{\partial \lambda}(\lambda^\ast, \boldsymbol{\gamma}^\ast) \neq 0,
  \]
  then there exists a neighborhood $V \subset \R^d$ of $\boldsymbol{\gamma}^\ast$ and a continuously differentiable function $\lambda: V \to \R$ such that $\lambda(\boldsymbol{\gamma}^\ast) = \lambda^\ast$ and for all $\boldsymbol{\gamma} \in V$
  \begin{subequations}
    \begin{align}
      \label{eqn:implicit_zero}
      q(\lambda(\boldsymbol{\gamma}), \boldsymbol{\gamma}) &= 0, \\
      \label{eqn:lam_gamma}
      \lambda(\boldsymbol{\gamma}) &= \lambda^\ast - \sum_{j=1}^d \frac{\frac{\partial q}{\partial \gamma_j}(\lambda^\ast, \boldsymbol{\gamma}^\ast)}{\frac{\partial q}{\partial \lambda}(\lambda^\ast, \boldsymbol{\gamma}^\ast)} (\gamma_j - \gamma^\ast_j) + o\left(\left\lVert \boldsymbol{\gamma} - \boldsymbol{\gamma}^\ast \right\rVert \right).
    \end{align}
  \end{subequations}
\end{Lemma}
\begin{proof}
  The Implicit Function Theorem delivers the existence of $V$ and $\lambda: V \to \R$ satisfying~\eqref{eqn:implicit_zero} and
  \begin{equation} \label{eqn:dlam_dgamma}
    \frac{\partial \lambda}{\partial \gamma_j}(\lambda(\boldsymbol{\gamma}), \boldsymbol{\gamma}) = \left. -\frac{\partial q}{\partial \gamma_j}(\boldsymbol{\gamma}) \middle/ \frac{\partial q}{\partial \lambda}(\lambda(\boldsymbol{\gamma}), \boldsymbol{\gamma}) \right. \quad \text{for all } j = 1, \dotsc, d.
  \end{equation}
  Taylor expansion of $\lambda(\boldsymbol{\gamma})$ around $\boldsymbol{\gamma}^\ast$ together with~\eqref{eqn:dlam_dgamma} yields~\eqref{eqn:lam_gamma}.
\end{proof}

  {The next lemma shows that all zeros of $U_{k+1}(\lambda, \boldsymbol{\gamma})$ are real and distinct for each combination of the
  parameters with $\gamma_k \in (0,1], \ k =1, \ldots, N$. This guarantees that the  hypothesis $\frac{\partial U_{k+1}(\lambda, \boldsymbol{\gamma})}{\partial \lambda }(\xi) \ne 0$, with
  $\xi$ a zero of $U_{k+1}(\lambda, \boldsymbol{\gamma})$, is always satisfied,
  allowing us to use Lemma \ref{lem:how_roots_move} to  show
  that the extremal values of the zeros of $U_{k+1}(x, \boldsymbol{\gamma})$
  are obtained at the boundary values of $\gamma_k$, namely either $0$ or $1$.
  \begin{Lemma}
	\label{lemma:real_roots}
	  {
	  Let $\gamma_k \in (0,1]$ for $k =1, \ldots, N$. Then the  polynomial $U_{k+1}(\lambda, \boldsymbol{\gamma})$ has
	  $k+1$ real and distinct roots.}
\end{Lemma}
}
\begin{proof}
	{
	We first show that the sequence of polynomials of degree $2k$, $a_{k+1}(\lambda, \boldsymbol \gamma) =  U_k(\lambda, \boldsymbol \gamma) U'_{k+1}(\lambda, \boldsymbol \gamma) -U'_k(\lambda, \boldsymbol \gamma)U_{k+1}(\lambda, \boldsymbol \gamma)$ is always positive 
	for every choice of $\boldsymbol \gamma$.
	{We first observe that the highest order monomial in $a_{k+1}(\lambda, \bgamma)$ is $x^k (k+1) x^k - k x^{k-1} x^{k+1} = x^{2k}$, which shows that $a_{k+1}$ is a monic polynomial.}
	We set, for brevity, $c_{k+1} = (-1)^{k+1}(1-\gamma_k)$, and omit the parameter $\boldsymbol \gamma$ for the rest of the proof.  
	From \eqref{un} we can write
	\begin{align*} U_{k+1}(\lambda) &= (c_{k+1} + \lambda)  U_k(\lambda) - \gamma_k U_{k-1}(\lambda), \\
	U'_{k+1}(\lambda)  &= U_k(\lambda) + (c_{k+1} + \lambda)  U'_k(\lambda) - \gamma_k U'_{k-1}(\lambda). \end{align*}
	Now, multiplying the first of the previous by $U'_{k}(\lambda)$ and the second by $U_{k}(\lambda)$ yields
	\begin{align*} U'_k(\lambda)U_{k+1}(\lambda) &= (c_{k+1} + \lambda)  U_k'(\lambda) U_k(\lambda) - \gamma_k  U_k'(\lambda) U_{k-1}(\lambda), \\
	U_k(\lambda) U'_{k+1}(\lambda)  &= U_k(\lambda)^2 + (c_{k+1} + \lambda)  U_k(\lambda) U'_k(\lambda) - \gamma_k U_k(\lambda) U'_{k-1}(\lambda). \end{align*}
	Taking the difference between the expressions provides
	\[	a_{k+1}(\lambda) = U_k(\lambda)^2 + \gamma_k a_k(\lambda).  \]
	Since $a_1 = 1$, by induction the sequence of polynomials $\{a_k(\lambda)\}$ is positive for all $\lambda \in \R$.
	}

	{We now proceed by induction to prove the statement, which we assume to be true for index $k$, namely that there are real numbers
	\[ \xi_1^{(k)} <  \xi_2^{(k)} <  \ldots <  \xi_k^{(k)}, \]
	such that $U_k(\xi_j^{(k)}) = 0, \ j = 1, \ldots, k$.
	Since $U_k(\lambda)$ is a monic polynomial it turns out that {$(-1)^i U_k'(\xi^{(k)}_{k-i}) > 0$ for $i = 0, \dotsc, k-1$.}
	For every $j$ it holds
	\[ 0 < a_{k+1}(\xi_j^{(k)}) =  U_k(\xi_j^{(k)}) U'_{k+1}(\xi_j^{(k)}) -U'_k(\xi_j^{(k)})U_{k+1}(\xi_j^{(k)}) =
	-U'_k(\xi_j^{(k)})U_{k+1}(\xi_j^{(k)}), \]
	which gives immediately the alternating signs of the values of the polynomial $U_{k+1}(\lambda)$ at the roots of $U_k(\lambda)$
	and $k+1$ intervals in which we may apply Bolzano's theorem. 
	}
	\end{proof}

	{The following Lemma states that at a a root $\xi$, the polynomial $U_{k+1}$ is frequently a monotonic function in all its parameters $\gamma_*$, and if a derivative of $U_{k+1}$ with respect to $\gamma_j$ vanishes, then $\xi$ is also a zero of a lower-order polynomial in the sequence.}
	{We again suppress the $\bgamma$-dependency in the arguments for better readability.}
	\begin{Lemma} 
		{Let $0 < j \le k$, $\gamma_i \in (0, 1]^N$ for $i = j, \dotsc, k$, and $\xi$ be a zero of $U_{k+1}$.}{Let $\xi$ be a zero of $U_{k+1}$ (for $j \le k$) {and $\gamma_i \in (0,1]$ (for $j \le i \le k $).}} Then either
		$\frac {\partial U_{k+1}}{\partial \gamma_j}(\xi)  \ne 0 $ or $\xi$ is a zero of a lower-order polynomial $U_s$ (with 
		$s \le k$). 
	\end{Lemma}
	\begin{proof}
		If $k = j$, the statement follows immediately from the definition of $U_{j+1}$: 
		\[\dfrac {\partial U_{j+1}}{\partial \gamma_j}(\xi)  = 
		   (-1)^j U_j(\xi) - U_{j-1}(\xi) = -\frac{1}{\gamma_j} \left(\xi + (-1)^{j+1} \right) U_j(\xi).  \]
		   {Since $U_j(\xi)$ cannot be zero, by virtue of the interlacing property shown in the proof of Lemma 4,}
		   the previous implies that $\frac {\partial U_{j+1}}{\partial \gamma_j}(\xi) = 0 \ \Leftrightarrow \ \xi = (-1)^j$.
		   If $j$ is even it holds that $U_1(\xi) = 0$, and if it is odd then $U_2(\xi, 0) = 0$.

		Now, to consider the remaining cases, we fix $0 < j < k$.
		If we define
		\[ \eta_k(x) = (-1)^{k+1}(1-\gamma_k) + x,\]
		we can write
		\[ 0 = U_{k+1}(\xi) = \eta_k(\xi) U_k(\xi) - \gamma_k U_{k-1}(\xi),\]
		from which (again $U_k (\xi)$ cannot be zero):
		\[ \eta_k(\xi) = \gamma_k \frac{U_{k-1}(\xi)}{U_{k}(\xi)}.\]
We now compute the relevant partial derivatives, with every polynomial evaluated at $\xi$:
		\[ \dfrac {\partial U_{k+1}}{\partial \gamma_j} = \eta_k \dfrac {\partial U_{k}}{\partial \gamma_j}
		- \gamma_k  \dfrac {\partial U_{k-1}}{\partial \gamma_j} 
		 = \frac{\gamma_k}{U_k}\left( U_{k-1} \dfrac {\partial U_{k}}{\partial \gamma_j} 
		- \dfrac {\partial U_{k-1}}{\partial \gamma_j}   U_k  \right). \]
		Notice that for $k>j$, $\eta_k$ does not depend on $\gamma_j$. 
		We now show that the sequence 
		\[ \widehat a_{k+1} = U_{k}(\xi) \dfrac {\partial U_{k+1}}{\partial \gamma_j} (\xi) -
		\dfrac {\partial U_{k}}{\partial \gamma_j} (\xi) U_{k+1}(\xi)\]
		can be computed by recurrence, as in the proof of Lemma 4, using the identities:
		 \begin{align*} U_{k+1} &= \eta_k U_k - \gamma_k U_{k-1}, \\
         \dfrac {\partial U_{k+1}}{\partial \gamma_j}  &= \eta_k  \dfrac {\partial U_{k}}{\partial \gamma_j} - \gamma_k  \dfrac {\partial U_{k-1}}{\partial \gamma_j}. \end{align*}
        Multiplying the first of the previous by $\frac {\partial U_{k}}{\partial \gamma_j}$ and the second by $U_{k}$, then taking the difference yields
		\[ \widehat a_{k+1} = \gamma_k \widehat a_k, \qquad \ldots, \qquad  \widehat a_{k} = \gamma_{k-1} \ldots \gamma_{j+1} \widehat a_{j+1}.\]
		We can therefore write
		\begin{align*}  \dfrac {\partial U_{k+1}}{\partial \gamma_j} &= \frac{1}{U_k}\gamma_k \ldots \gamma_{j+1}   \left(U_{j} \dfrac {\partial U_{j+1}}{\partial \gamma_j} -
		\dfrac {\partial U_{j}}{\partial \gamma_j}  U_{j+1}\right)
			= \frac{1}{U_k}\gamma_k \ldots \gamma_{j+1} U_{j} \dfrac {\partial U_{j+1}}{\partial \gamma_j} \\
			&= \frac{1}{U_k}\gamma_k \ldots \gamma_{j+1} U_{j} \left((-1)^j U_j - U_{j-1}\right),
		\end{align*}
		{where we use that $U_j$ does not depend on $\gamma_j$ and hence $\frac {\partial U_{j}}{\partial \gamma_j} = 0$}. We conclude that 
		$\frac {\partial U_{k+1}}{\partial \gamma_j}(\xi) = 0$ occurs if either $U_j(\xi) = 0$ or $(-1)^j U_j(\xi) - U_{j-1}(\xi) = 0$. 
		In this last case we also have that
		\[ U_{j+1}(\xi) = \left( (-1)^{j+1} + \xi \right) U_j(\xi) + \gamma_j\left((-1)^j U_j(\xi) - U_{j-1}(\xi) \right) =  \left( \xi + (-1)^{j+1} \right) U_j(\xi).\]
		We consider separately the cases of $j$ even and odd. Assuming first that $j$ is even, then $U_{j+1}(\xi) = (\xi-1) U_j(\xi) = U_1(\xi) U_j(\xi)$, and 
		\[ U_{j+2}(\xi) = ((1-\gamma_{j+1} + \xi) U_{1}(\xi) - \gamma_{j+1}) U_j(\xi) = U_2(\xi,\gamma_{j+1}) U_j(\xi).\]
		Exploiting the recurrence relation \eqref{un}, we obtain
		\[ 0 = U_{k+1}(\xi) =  U_{k+1-j}(\xi,\gamma_{j+1}, \ldots, \gamma_{k}) U_j(\xi).\]
		If $j$ is odd, then $U_{j+1}(\xi) = (\xi+1) U_j(\xi) = \dfrac{U_2(\xi, 0)}{\xi-1}
		U_j(\xi)$. Again, by exploiting the recurrence \eqref{un} we obtain 
		\[ 0 = U_{k+1}(\xi) =  {\frac{1}{\xi - 1}} U_{k+2-j}(\xi,0,\gamma_{j+1}, \ldots, \gamma_{k}) U_j(\xi).\]
		In both cases {it holds that if} $\dfrac {\partial U_{k+1}}{\partial \gamma_j}(\xi)$ vanishes then $\xi$ is a zero of a
		lower-order polynomial $U_s$ for a specific choice of its parameters.
	\end{proof}
	Summarizing, we have proved that either
	\begin{enumerate}
		\item[(i)] any zero $\xi$ of the polynomial $U_k(\lambda, \bgamma)$ is a monotonic function in each of its parameters
		$\gamma_j \in (0,1]$. If this function is increasing (respectively decreasing), its maximum (respectively minimum) 
		is achieved at $\gamma_j = 1$. Using a continuity argument, we can use $\gamma_j = 0$ to obtain a lower (respectively upper) bound of the value of $\xi$, or
	\item[(ii)]
		$\xi$ is also a zero of a suitable polynomial $U_s$ of lower order, {the value of which does not depend on $\gamma_j$}.
	\end{enumerate}

	\subsection{Zeros of $U_{k+1}(x, \boldsymbol{\gamma})$}
We are now in a position to establish ranges for the zeros of the polynomial family $U_*$, and hence deploy Theorem \ref{theo2} to obtain eigenvalue bounds of $\mathcal{P}_D^{-1} \mathcal{A}$ for any number of blocks.

	{
	\begin{Theorem}
		\label{Theo1}
		The zeros of $U_{k+1}(x, \boldsymbol{\gamma})$ are contained in the following set:
		\begin{equation*}
			\mathcal{I}_{k+1} \equiv \bigcup_{j=1}^{k+1} \mathcal{I}(P_{j})   \cup -\mathcal{I}(P_{j-1}),  
        \end{equation*}
		where $P_j$ is as in \Cref{defTV}.
		\label{Th1}
	\end{Theorem}}

	\begin{proof}
		The findings of \Cref{lemma:real_roots} and the previous observations allow us to choose $\gamma_j \in \{0,1\}$, $j = 1, \ldots, k$. {Notice that, whenever $\frac {\partial U_{k+1}}{\partial \gamma_j}(\xi) = 0$,  with $\xi$ a zero of $U_{k+1}$, we may fix $\gamma_j = 0$ for the analysis}.
		
		If $\boldsymbol \gamma = \boldsymbol 1_k$ then $U_{k+1} \equiv P_{k+1}$ and
		the thesis immediately follows. Otherwise,
		we show by induction that for every other choice
		of $\boldsymbol \gamma \in \{0,1\}^k$ the zeros of $U_{k+1}$ are in {$\bigcup_{j=1}^{k+1} \mathcal{I}(P_{j})\cup \mathcal{I}(V_{j})$.} 
		Let us assume the statement is true for indices $j \le k$ and prove it for index $k+1$.
        
		The basis of the induction: the roots of $U_2(x,\boldsymbol \gamma)$ {are the roots of $x^2-1 = 0$ if $\gamma_1 = 0$, or
		those of $x^2 -x - 1$ if $\gamma_1 = 1$, which once put together form the set
		$\{\pm 1, \frac{1\pm \sqrt 5}{2}\}$. Employing our convention of combining negative and positive intervals gives that $\mathcal{I}(P_{2})\cup \mathcal{I}(V_{2}) \equiv [-1, \frac{1 - \sqrt 5}{2}] \cup [1, \frac{1 + \sqrt 5}{2}]$.}
        
		Denote now $i = \max\{r \text{ such that } \gamma_r = 0\}$.
		Then for $j > 0$ we have (since $\gamma_{i+j} = 1$):
		\[U_{j+i+1} = x U_{j+i}- U_{j+i-1}.\]
		We then show by induction that
		\begin{equation} \label{Q}
		U_{j+i} = Q_j U_{i},\quad  j = 0, \ldots, k-i+1, \end{equation}
		where $Q_j$ satisfies the same recurrence relation
		as the polynomials {$P_*$} and $V_*$:
			\[ Q_{j+2} = x Q_{j+1} - Q_{j}, \quad j \ge 0. \]
		First, \eqref{Q} holds trivially for $j=0$, if we set $Q_0 = 1$.
		Since $\gamma_i = 0$,
		\[
    U_{i+1} =
    \begin{cases}
      (x+1) U_{i} \equiv Q_1^{\text{(even)}} U_i & \text{if $i+1$ is even}, \\
      (x-1) U_{i} \equiv Q_1^{\text{(odd)}} U_i & \text{if $i+1$ is odd}.
    \end{cases}
    \]
		Then in both cases, with $j \ge 0$,
		\[U_{j+i+2} = x U_{j+i+1}- U_{j+i} = x Q_{j+1} U_i - Q_j U_i = (x Q_{j+1} - Q_j) U_i = Q_{j+2} U_i, \]
		which proves \eqref{Q}.
		Considering the case $i+1$ odd, we have  that
		$Q_j^{\text{(odd)}} \equiv {P_j}$. If $i+1$ is even, we have that
		$Q_0 = Q_0^{\text{(even)}} = \dfrac{V_{1}}{x-1}$,
		$Q_1^{\text{(even)}} = \dfrac{V_{2}}{x-1}$ and hence
		$Q_j^{\text{(even)}} \equiv \dfrac{V_{j+1}}{x-1}$ (notice that we may argue this even for $x=1$ since, by \eqref{Vrecur}, $V_j(x)/(x-1)$ is a polynomial). Recalling \eqref{Vrecur} {and letting $j=k-i+1$ in~\eqref{Q}}, we conclude that
		\[ U_{k+1}(x) = \begin{cases}
			{P_{k-i+1}(x)} U_i(x) & \text{if } i+1 \text{ is } \text{odd,} \\
		\dfrac{V_{k-i+2}(x)}{x-1} U_i(x)  = (-1)^{k-i+1} {P_{k-i+1}(-x)} U_i(x)& \text{if } i+1 \text{ is } \text{even.} \end{cases}
			\]
			Applying the inductive hypothesis to $U_i$ (since $i \le k$) completes the proof.
	\end{proof}
	The exact zeros of the polynomials $U_k(\lambda, \boldsymbol \gamma)$ for binary $\bgamma$ can also be derived from the eigenvalues of suitably chosen tridiagonal matrices:
\begin{Theorem} \label{thm:zeros_of_U}
  Let $\bgamma \in \{0, 1\}^k$ and denote by $j_1 < \ldots < j_{\ell-1}$  
	all indices where $\gamma_{j}$ vanishes. Moreover, define $k_i = j_i - j_{i-1}$ for $i = 1, \dotsc, \ell$, where $j_0 = 0$ and $j_\ell = k+1$.
  Then, the zeros of $U_{k+1}(x, \bgamma)$ are given by 
  \[
  x_{i,s} = 2 (-1)^{j_{i-1}} \cos \left(\frac{(2s - 1) \pi}{2 k_i + 1}\right) \quad \text{for } s = 1, \dotsc, k_i, ~ i = 1, \dotsc, \ell.
  \]
\end{Theorem}
\begin{proof}
  Let us define a symmetric tridiagonal $k \times k$ matrix $M_k$ with diagonal entries $c_1, \dotsc, c_{k}$ and sub- and super-diagonal entries $b_2, \dotsc, b_{k}$. The characteristic polynomial $\chi_k(\lambda)$ of $M_k$ is obtained by the well-known recurrence relation~\cite{wilkinson1962handbook}:
  \[
  \chi_0(\lambda) = 1, \qquad \chi_1(\lambda) = c_1 - \lambda, \qquad \chi_{j+1} = (c_{j+1} - \lambda) \chi_j(\lambda) - b_{j+1}^2 \chi_{j-1}(\lambda), \quad j = 1, \dotsc, k-1.
  \]
  If we choose $c_1 = -1$, $c_j = (-1)^j (1 - \gamma_{j-1})$ for $j = 2, \dotsc, k$, and $b_j = \sqrt{\gamma_{j-1}}$ for $j = 2, \dotsc, k$, then $U_{k+1}(x, \bgamma)$ coincides with $\chi_{k+1}(-x)$. Hence, the roots of $U_{k+1}(x, \bgamma)$ are the eigenvalues of $-M_{k+1}$.

  For $\gamma_j = 0$ we have $b_{j+1} = 0$ and the matrix $M_{k+1}$ decouples into two smaller blocks.
	Hence, because $\bgamma \in \{ 0, 1 \}^k$, the matrix $M_{k+1}$ can be decomposed into $\ell$ blocks of sizes $r = k_1, \dotsc, k_\ell$ of the form
  \[
    M_r^- =
    \begin{bmatrix}
      -1 & 1 \\
      1 & 0 & \ddots \\
      & \ddots & \ddots & 1 \\
      & & 1 & 0
    \end{bmatrix}
    \quad \text{or} \quad
    M_r^+ =
    \begin{bmatrix}
      1 & 1 \\
      1 & 0 & \ddots \\
      & \ddots & \ddots & 1 \\
      & & 1 & 0
    \end{bmatrix},
  \]
  where the leading entry is $c_{j_{i-1}+1} = (-1)^{j_{i-1}+1}$. 
  We can view $M_r^+$ and $- M_r^-$ as multiple saddle-point systems of dimension $r$, and their analogous block diagonal preconditioners are identity matrices in both cases. Hence, the eigenvalues of $M_r^+$ are $\{ 1 \}$ if $r = 1$ or $\{ 2 \cos ( (2s - 1) \pi / (2 r + 1)) \mid s = 1, \dotsc, r \}$ if $r > 1$,
	while the eigenvalues of $M_r^-$ are the same up to a negative sign, which yields the asserted formulas.
\end{proof}
\begin{Remark}
  Theorem \ref{Theo1}, which we proved using the recursive definition of polynomials, could also be proved via Theorem \ref{thm:zeros_of_U}, leveraging a linear algebra-oriented argument. We believe it to be worthwhile to present both approaches here.
\end{Remark}
	\begin{Remark} The bounds for $k=2$ coincide with the eigenvalues of the $3 \times 3$ block systems in \cite{Bradley}, and the bounds for $k=3$ with those of
the $4 \times 4$ block systems in \cite[Thm. 3.1]{pearson2023symmetric}. Figure \ref{fig:Bounds}, which reports the intervals containing
		the eigenvalues of the preconditioned multiple saddle-point matrices, for $k = 1, \ldots, 9$, extends Table 1 in \cite{pearson2023symmetric}.
	\end{Remark}



	\begin{figure}
	\begin{minipage}{7cm}
	\begin{tabular}{|c||c|c|c|c|} \hline
		$k$  & $\text{Bound}_l^-$  & $\text{Bound}_l^+$  & $\text{Bound}_u^-$  & $\text{Bound}_u^+$  \\ \hline \hline
		9  &   $-1.9727$   &  $-0.1495$   &   0.1652  &    1.9777 \\[.2em]
		8  &   $-1.9659$   &  $-0.1845$   &   0.1652  &    1.9727 \\[.2em]
		7  &   $-1.9563$   &  $-0.1845$   &   0.2091  &    1.9659 \\[.2em]
		6  &   $-1.9419$   &  $-0.2411$   &   0.2091  &    1.9563 \\[.2em]
		5  &   $-1.9190$   &  $-0.2411$   &   0.2846  &    1.9419 \\[.2em]
		4  &   $-1.8794$   &  $-0.3473$   &   0.2846  &    1.9190 \\[.2em]
		3  &   $-1.8019$   &  $-0.3473$   &   0.4450  &    1.8794 \\[.2em]
		2  &   $-1.6180$   &  $-0.6180$   &   0.4450  &    1.8019 \\[.2em]
		1  &   $-1$   &  $-0.6180$   &   1  &    1.6180 \\[.0001em] \hline
	\end{tabular}
	\end{minipage}
	\begin{minipage}{7cm}
		\includegraphics[width=8cm]{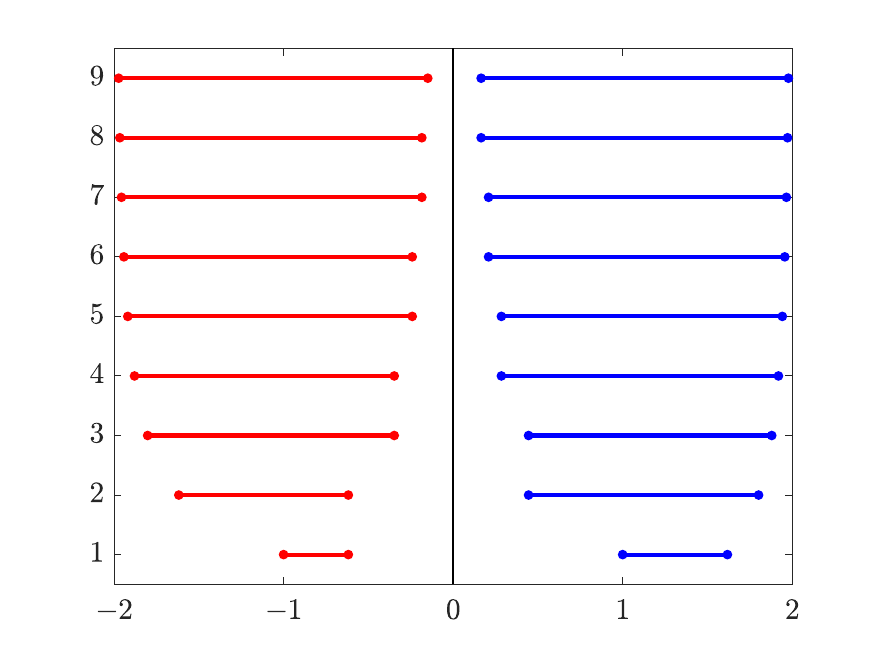}
	\end{minipage}
		\caption{Bounding intervals $\mathcal I_{k+1} = [\text{Bound}_l^-, \text{Bound}_l^+] \cup [\text{Bound}_u^-, \text{Bound}_u^+]$ for $U_{k+1}$, $k = 1, \ldots, 9$. This corresponds to eigenvalue bounds for preconditioned multiple saddle-point systems with $k+1$ blocks (i.e., $N=k$).}
		\label{fig:Bounds}
	\end{figure}

\subsection{Numerical validation}
To validate the analysis of this section, in particular Theorem \ref{Theo1}, we now construct randomly-generated test cases with increasing numbers of blocks\footnote{All computational tests in this paper are carried out on an Intel(R) Core(TM) i5-13500T CPU with 14 cores and 20 threads, in {\scshape Matlab} R2022a.}. For each numerical test with $N+1$ blocks, we set $n_0 = 300$ and let $n_{k+1} = n_k - \lceil 10 \ast \texttt{rand} \rceil$ for $k = 0, \ldots, N-1$, using {\scshape Matlab}'s \texttt{rand} function repeatedly. We also evaluate $B_k$ using \texttt{rand} for each matrix entry. We carry out tests within two settings for $A_k$:
\begin{itemize}
    \item Compute a diagonal matrix $A_0$ using \texttt{rand} for each diagonal entry. For $k=1, \ldots, N$, compute diagonal matrices $A_k$ using $10^{-4} \ast \texttt{rand}$ for each diagonal entry.
    \item Compute $A_0$ using \texttt{rand} for each entry, take its symmetric part, then add 1.01 times an identity matrix multiplied by the absolute value of the smallest eigenvalue, to ensure
    symmetric positive definiteness. For $k=1, \ldots, N$, compute $A_k$ using \texttt{rand} for each entry, take its symmetric part, then add an identity matrix multiplied by the absolute value of the smallest eigenvalue, to ensure
    symmetric positive semi-definiteness.
\end{itemize}


 \begin{table}[h!]
 \begin{center}
 \begin{tabular}{|r|r||c|c|c|c|c|} \hline
	 $N$ (\text{size}) &&  $\text{Bound}_l^-$ &  $\text{Bound}_u^-$ &  $\text{Bound}_l^+$ &  $\text{Bound}_u^+$  & Iter \\
	 \hline 
	 \hline 
	 1 (595) & computed          & $-0.6207$    & $-0.6180$     & 1.0000   &  1.6180   &    10 \\
	 1 (591) & computed          & $-0.9999$    & $-0.6180$     & 1.0000   &  1.6180   &    33 \\
	1 &theoretical    & $-1$    & $-0.6180$     & 1    &  1.6180 &  \\
\hline
	 2 (878)  & computed          & $-1.2486$    & $-0.6180$     & 0.4450     & 1.8019    &   38 \\
	 2 (886)  & computed          & $-1.5313$    & $-0.6251$     & 0.5371     & 1.7102    &   59 \\
	2 &theoretical    & $-1.6180$    & $-0.6180$     & 0.4450     & 1.8019      &  \\
\hline
	 3 (1171) & computed          & $-1.5321$    & $-0.3473$     & 0.4451     & 1.8794      &63 \\
	 3 (1175) & computed          & $-1.6421$    & $-0.4951$     & 0.6095     & 1.7298      &64 \\
	3 &theoretical    & $-1.8019$    & $-0.3473$     & 0.4450     & 1.8794      &  \\
\hline
	 4 (1408) & computed          & $-1.6825$    & $-0.3495$     & 0.2846     & 1.9190      &131 \\
	 4 (1443) & computed          & $-1.6483$    & $-0.5410$     & 0.4993     & 1.7186      &69 \\
	4 &theoretical    & $-1.8794$    & $-0.3473$     & 0.2846     & 1.9190      &  \\
\hline
	 5 (1696) & computed          & $-1.7709$    & $-0.2411$     & 0.2852     & 1.9419      &162 \\
	 5 (1742) & computed          & $-1.6768$    & $-0.4830$     & 0.5486     & 1.7222      &72 \\
	5 &theoretical    & $-1.9190$    & $-0.2411$     & 0.2846     & 1.9419      &  \\
\hline
	 6 (1910) & computed          & $-1.8271$    & $-0.2559$     & 0.2091     & 1.9563      &219 \\
	 6 (1966) & computed          & $-1.6732$    & $-0.5285$     & 0.4815     & 1.7295      &72 \\
	6 &theoretical    & $-1.9419$    & $-0.2411$     & 0.2091     & 1.9563      &  \\
\hline
	 7 (2183) & computed          & $-1.8649$    & $-0.1845$      &0.2103     & 1.9659      &255 \\
	 7 (2228) & computed          & $-1.6793$    & $-0.4792$      &0.5339     & 1.7211      &74 \\
	7 &theoretical    & $-1.9563$    & $-0.1845$      &0.2091     & 1.9659      &  \\
\hline
	 8 (2549) & computed          & $-1.8916$    & $-0.2027$      &0.1652     & 1.9727      &310 \\
	 8 (2543) & computed          & $-1.6846$    & $-0.5136$      &0.4679     & 1.7282      &75 \\
	8 &theoretical    & $-1.9659$    & $-0.1845$      &0.1652     & 1.9727      &  \\
\hline
 \end{tabular}
 \end{center}
 \medskip
 \caption{Endpoints of the intervals of negative and positive eigenvalues of preconditioned multiple saddle-point system, number of {\scshape Minres} iterations required to solve a system with a random right-hand side (denoted `Iter'), as well as the value of $N$ (number of blocks minus one) and the dimension of the system solved. For each value of $N$ we solve one problem with diagonal $A_k$ (top of cell), one problem with non-diagonal random matrices (middle), and also state the bounds predicted by Theorem \ref{Theo1}.}
	 \label{tabexact}
	 \end{table}

The computed (negative and positive) ranges of eigenvalues, as well as the theoretically-computed bounds, are shown in Table \ref{tabexact}. Also reported are the number of {\scshape Minres} iterations required to reduce the residual norm by a factor of $10^{10}$, for a system with a right-hand side generated using the \texttt{rand} function for each entry.

In the first experimental setup we notice the close agreement of the theoretical bounds with the intervals containing the computed eigenvalues, aside from the bound on the smallest negative eigenvalue. As predicted by the theoretical and computed bounds, the {\scshape Minres} iteration count noticeably increases as $N$ does. 
In the second setup this bound is more accurately captured for small $N$, but as $N$ increases all the theoretical bounds become looser, and as a result there is a much less substantial increase in {\scshape Minres} iterations.

\section{Multiple saddle-point linear systems: The inexact case}
\label{sec:inexact}

In practice, the use of exact Schur complements $S_i$ is often prohibitive, especially for larger $i$ due to the recursive definition of the Schur complements based on inverses of previous Schur complements.
In this section we present two possible approaches for analyzing the effect of applying the Schur complements in $\mathcal{P}_D$ only approximately. We first provide a general qualitative perturbation result for the case of general $N$ in Section \ref{sec:Perturb}, and then provide quantitative bounds based on field-of-value indicators in the specific setting of double saddle-point systems in Section \ref{sec:DSPAnalysis}.

\subsection{Perturbation analysis} \label{sec:Perturb}

For our analysis, we denote the approximated Schur complements by $\widehat{S}_i$ for $i = 0, \dotsc, N$. It is important to point out that each of these should approximate a perturbed Schur complement $\widetilde{S}_i$ that takes the previous Schur complement approximations into account according to
\begin{equation} \label{ApproxSi}
  \widehat{S}_0 \approx \widetilde{S}_0 := S_0 = A_0, \qquad
  \widehat{S}_i \approx \widetilde{S}_i := A_i + B_i \widehat{S}_{i-1}^{-1} B_i^{\top}, \quad i = 1, \dotsc, N.
\end{equation}

To analyze the approximate preconditioner
\[
\widehat{\mathcal{P}}_D = \text{blkdiag}\left(\widehat{S}_0, \dotsc, \widehat{S}_N\right),
\]
we define the perturbed block matrix, 
\[
\widehat{\mathcal{A}} = \mathcal{A} + \text{blkdiag}\left(\Delta_0, \ldots, \Delta_N\right), \qquad \text{with} \quad \Delta_k = (-1)^{k} \left(\widehat{S}_k - \widetilde{S}_k\right) ,
\]
and denote its diagonal blocks by $\widehat{A}_i$ for $i = 0, \dotsc, N$.
We shall prove the following perturbation result:
\begin{Theorem} \label{thm:perturbation}
	 {
Let $\widehat{A}_k$ be symmetric positive semi-definite (with $\widehat{A}_0$ symmetric positive definite) and $\widehat{S}_k$ be symmetric positive definite, for $k=0, \dotsc, N$, and let $\sigma^-$ and $\sigma^+$ denote the minimum and maximum 
over the eigenvalues of all $-\widehat{S}_k^{-1} \Delta_k = (-1)^{k+1} \left(I - \widehat{S}_k^{-1} \widetilde{S}_k\right)$, which are all real. Then, the eigenvalues of the approximately preconditioned matrix $\widehat{\mathcal{P}}_D^{-1} \mathcal{A}$ are contained in the Min\-kows\-ki sum $\mathcal{I}_{N+1} + [\sigma^-, \sigma^+]$.
}
\end{Theorem}
We present two lemmas before we prove Theorem \ref{thm:perturbation}, the idea for the proof of which is based on a backward analysis argument that asks: For which system matrix is the approximate preconditioner an ideal Schur complement preconditioner, to which Theorem \ref{theo2} applies?
\begin{Lemma} \label{lem:SCs_of_perturbed_A}
  The exact Schur complements of $\widehat{\mathcal{A}}$ are $\widehat{S}_i$ for $i=0, \dotsc, N$.
\end{Lemma}
\begin{proof}
	We proceed by induction. For $i=0$, we observe that $\widehat{A}_0 = A_0 + \widehat{S}_0 - \widetilde{S}_0 = \widehat{S}_0$. For $i > 0$, we exploit the inductive hypothesis ($\widehat{S}_{i-1}$ is an exact Schur complement of $\widehat{\mathcal{A}}$) and the definitions of $\widehat{A}_i$ and $\widetilde{S}_i$ to obtain for the $i$-th Schur complement of $\widehat{\mathcal{A}}$ that
	\[
		\widehat{A}_i + B_i \widehat{S}_{i-1}^{-1} B_i^{\top}
= A_i + \widehat{S}_i - \widetilde{S}_i + B_i \widehat{S}_{i-1}^{-1} B_i^{\top}
= A_i + \widehat{S}_i - A_i - B_i \widehat{S}_{i-1}^{-1} B_i^{\top} + B_i \widehat{S}_{i-1}^{-1} B_i^{\top} = \widehat{S}_i. \qedhere
\]
\end{proof}
\begin{Lemma}[Weyl's inequality] \label{lem:Weyl}
  Let $M_1$ and $M_2$ be Hermitian $n$-by-$n$ matrices. 
  Denote their eigenvalues by $\lambda_1(\cdot) \ge \ldots \ge \lambda_n(\cdot)$. Then
	\begin{align*}
		\lambda_{i+j-1}(M_1 + M_2) &\le \lambda_i(M_1) + \lambda_j(M_2), \qquad  i+j \le n+1,\\
		\lambda_i(M_1) + \lambda_j(M_2) &\le \lambda_{i+j-n}(M_1 + M_2), \quad\ \, \, i+j \ge n+1.
	\end{align*}
\end{Lemma}
\begin{proof}
  A proof based on the Courant--Fischer min--max theorem can be found in \cite[Sec. 1.3.3]{tao2012topics}.
\end{proof}
\noindent \emph{Proof of Theorem \ref{thm:perturbation}.}
The approximately preconditioned matrix can be written as
{
\[
\widehat{\mathcal{P}}_D^{-1} \mathcal{A}
= \widehat{\mathcal{P}}_D^{-1} \widehat{\mathcal{A}}
- \text{blkdiag}\left(\widehat{S}_0^{-1} \Delta_0, \dotsc, \widehat{S}_{N}^{-1} \Delta_N\right).
\]
A similarity transformation with $\widehat{\mathcal{P}}_D^{1/2}$ yields the spectrally equivalent symmetric matrices
\[
\widehat{\mathcal{P}}_D^{-1/2} \mathcal{A} \widehat{\mathcal{P}}_D^{-1/2}
= \widehat{\mathcal{P}}_D^{-1/2} \widehat{\mathcal{A}} \widehat{\mathcal{P}}_D^{-1/2}
- \text{blkdiag}\left(\widehat{S}_0^{-1/2} \Delta_0 \widehat{S}_0^{-1/2}, \dotsc, \widehat{S}_{N}^{-1/2} \Delta_N \widehat{S}_N^{-1/2}\right).
\]
}
We can now apply Weyl's inequality. To find the relevant pairs of $i$ and $j$, we recall the decomposition
\[
{\mathcal{A}} =
\begin{bmatrix}
  I \\
  B_1 {S}_0^{-1} & -I \\
  & \ddots & \ddots \\
  & & B_N {S}_{N-1}^{-1} & (-1)^N I
\end{bmatrix}
\begin{bmatrix}
  {S}_0 \\
  & -{S}_1 \\
  & & \ddots \\
  & & & (-1)^N {S}_N
\end{bmatrix}
\begin{bmatrix}
  I & {S}_0^{-1} B_1^\top \\
  & -I & \ddots \\
  & & \ddots & {S}_{N-1}^{-1} B_N^\top \\
  & & & (-1)^N I
\end{bmatrix}.
\]
Sylvester's law of inertia now tells us that ${\mathcal{A}}$ and hence also $\widehat{\mathcal{P}}_D^{-1/2} {\mathcal{A}} \widehat{\mathcal{P}}_D^{-1/2}$ have $n^+ = \sum_{k \text{ even}} n_k$ positive, $n - n^+$ negative, and no zero eigenvalues, where $n = \sum_{k=1}^N n_k$.

Consequently, the indices of the extremal eigenvalues of the symmetrically preconditioned matrix $M_1 + M_2$ with $M_1 = \widehat{\mathcal{P}}_D^{-1/2} \widehat{\mathcal{A}} \widehat{\mathcal{P}}_D^{-1/2}$ and 
{
	\[M_2 = -\text{blkdiag}\left(\widehat{S}_0^{-1/2} \Delta_0 \widehat{S}_0^{-1/2}, \dotsc, \widehat{S}_{N}^{-1/2} \Delta_N \widehat{S}_N^{-1/2}\right)
 = \text{blkdiag}\left(\widehat{S}_0^{-1/2} \widetilde S_0 \widehat{S}_0^{-1/2}- I, 
 I - \widehat{S}_1^{-1/2} \widetilde S_1 \widehat{S}_1^{-1/2}, \dotsc\right) \]
}
are $1$ (largest positive), $n^+$ (smallest positive), $n^+ + 1$ (largest negative), and $n$ (smallest negative).

For the largest positive eigenvalue, Weyl's inequality for $i + j - 1 = 1$ provides the upper bound
\[
\lambda_1(M_1 + M_2) \le \lambda_1(M_1) + \lambda_{1}(M_2).
\]
For the smallest negative eigenvalue ($i + j - n = n$) we obtain
\[
\lambda_n(M_1 + M_2) \ge \lambda_{n}(M_1) + \lambda_{n}(M_2).
\]
We now move to the eigenvalues closest to zero. For the smallest positive eigenvalue, Weyl's inequality ($i + j - n = n^+$) delivers the lower bound
\[
\lambda_{n^+}(M_1 + M_2) \ge \lambda_{n^+}(M_1) + \lambda_{n}(M_2).
\]
Likewise, we obtain for the largest negative eigenvalue ($i + j - 1 = n^+ + 1$)
\[
\lambda_{n^+ + 1}(M_1 + M_2) \le \lambda_{n^+ + 1}(M_1) + \lambda_{1}(M_2).
\]
Noting that the eigenvalues of $M_1$ can be characterized by Theorem \ref{theo2} finishes the proof. \hfill \qed

\subsection{Refined bounds for the double saddle-point systems with inexact Schur complements}\label{sec:DSPAnalysis}

We now consider bespoke eigenvalue bounds when preconditioning the linear system in the case $N=2$:
\begin{equation*}
        \mathcal{A}w = b, \qquad \text{where} \quad \mathcal{A}  =
        \begin{bmatrix}
        A_0 & B_1^{\top} & 0\\
        B_1 & -A_1 & B_2^{\top}\\
0 & B_2 & A_2\end{bmatrix}.
\end{equation*}
We define the Schur complements as in \eqref{ApproxSi} to form an inexact block diagonal preconditioner $\widehat{\mathcal{P}}_D=\text{blkdiag}(\widehat S_0,  \widehat S_1, \widehat S_2)$, with $\widehat S_0$, $\widehat S_1$, $\widehat S_2$ symmetric positive definite.

We analyze the eigenvalue distribution  of the preconditioned matrix
$\widehat{\mathcal{P}}_D^{-1}\mathcal{A}$ and relate its spectral properties
to the extremal eigenvalues of
$\overline{S}_i \equiv \widehat S_i^{-1/2} S_i \widehat S_i^{-1/2}$, for $i=0, 1, 2$.
Finding the eigenvalues of $\widehat{\mathcal{P}}_D^{-1}\mathcal{A}$ is equivalent to solving
$$\mathcal{Q}_{\text{in}} \equiv \widehat \P_D^{-1/2} \mathcal A \widehat \P_D^{-1/2} v = \lambda v,
\qquad v = \begin{bmatrix} x\\y\\z\end{bmatrix},$$ 
	where $x$, $y$, and $z$ denote vectors of length $n_0$, $n_1$, and $n_2$, respectively.
Exploiting the block components of this generalized eigenvalue problem, we obtain
\begin{equation*}
	\begin{bmatrix}E_0  &  R_1^\top  &  0\\
		R_1  &  -E_1  &  R_2^\top\\
	0  &  R_2  & E_2 \end{bmatrix}
		\begin{bmatrix}x\\y\\z  \end{bmatrix}=\lambda
\begin{bmatrix}x\\y\\z  \end{bmatrix}.
\end{equation*}
where \begin{align*} E_i&=\widehat S_i^{-1/2}A_i\widehat S_i^{-1/2}, \quad i = 0, 1, 2,\\
	R_i &=\widehat S_i^{-1/2}B_i\widehat S_{i-1}^{-1/2}, \quad i = 1, 2.
\end{align*}
Notice that
	\begin{equation*}
    R_i R_i^\top = \widehat S_i^{-1/2} \left(\widetilde S_i - A_i\right) \widehat S_i^{-1/2} =
	\overline{S}_i - E_i. 
    \end{equation*}
We make use of the following indicators:
        \begin{align*}
		\label{indicators}
		\alpha_E^{(i)}  & \equiv \lambda_{\min} (E_i),  & \beta_E^{(i)}  & \equiv \lambda_{\max} (E_i), & \gamma_E^{(i)}(w_{k_i})
		&=\frac{w_{n_i}^\top E_i w_{n_i}}{w_{n_i}^\top w_{n_i}}  \in [ \alpha_{E_i},  \beta_{E_i} ] \equiv \widehat{\mathcal{I}}_{E_i},  \quad i = 0, 1, 2,  \nonumber\\[-.9em]
		&& \\[-.9em]
		\alpha_R^{(i)}  & \equiv \lambda_{\min} (R_i R_i^\top),  & \beta_R^{(i)}  & \equiv \lambda_{\max} (R_i R_i^\top), & \gamma_R^{(i)}(w_{n_i})
		&= \frac{w_{n_i}^\top R_i R_i^\top w_{n_i}}{w_{n_i}^\top w_{n_i}} \in [ \alpha_{R_i},  \beta_{R_i} ] \equiv \widehat{\mathcal{I}}_{R_i},  \quad i = 1, 2, \nonumber
\end{align*}
In the following, to make the notation easier, we remove the argument $w_*$ whenever one of the previous indicators is used.
		Finally, we define the vectors
		\[\boldsymbol{\gamma}_E = \begin{bmatrix} \gamma_E^{(0)}, \gamma_E^{(1)}, \gamma_E^{(2)} \end{bmatrix}, \quad
			\boldsymbol{\gamma}_R = \begin{bmatrix} \gamma_R^{(1)}, \gamma_R^{(2)} \end{bmatrix}.  \]


\subsection*{Spectral analysis}
The eigenvalue problem \eqref{dpsp} then reads
\begin{align}
	(E_0-\lambda I)x  &~+& R_1^\top y                &  &                           &~=& 
    0, \nonumber\\
	R_1x                    &~-& (E_1 + \lambda I) y &~+ & R_2^\top z                &~=& 0, \label{Eq33.1}\\
				& & R_2y                      &~+ & (E_2  - \lambda I) z&~=& 0. \label{Eq34.1}
\end{align}
We first concentrate on a classical saddle-point linear system. Letting
	\[\mathcal{A}_{SP} =
\begin{bmatrix}
	A_0 & B_1^{\top} \\
	B _1& -A_1\\
\end{bmatrix}, \qquad
	\widehat{\mathcal{P}}_{SP} =
\begin{bmatrix}
	\widehat{S}_0 & 0 \\
	0 & \widehat{S}_1\\
\end{bmatrix},  \]
		then the eigenvalues of
$\widehat{\mathcal{P}}_{SP}^{-1}\mathcal{A}_{SP}$ are the same as those of
\begin{equation}\label{Eq31.0}
        \begin{bmatrix}E_0  &  R_1^\top  \\
	R_1 &  - E_1  \\ \end{bmatrix}
                \begin{bmatrix}x\\y  \end{bmatrix}=\lambda
\begin{bmatrix}x\\y \end{bmatrix}.
\end{equation}

			The following theorem characterizes the eigenvalues of the preconditioned matrix $\mathcal{A}_{SP}$,
			and hence a classical saddle-point linear system preconditioned by ${\widehat{\mathcal{P}}_{SP}}$, in terms
			of the previously defined indicators.
The findings of this theorem also constitute the basis for the proof of \Cref{Theo:pi}.
\begin{Theorem} \label{Theo:p}
	\label{saddle-point}
	The eigenvalues of $\widehat{\mathcal{P}}_{SP}^{-1}\mathcal{A}_{SP}$  are either contained in $\widehat{\mathcal{I}}_{E_0}$, or they
	are the roots of the $(\boldsymbol \gamma_E, \boldsymbol \gamma_R)$-parametric family of polynomials
	\begin{equation*}
		p(\lambda; \boldsymbol \gamma_E, \boldsymbol \gamma_R) \equiv \lambda^2 - \lambda(\gamma_E^{(0)}  - \gamma_E^{(1)})  -\gamma_R^{(1)} -\gamma_E^{(0)} \gamma_E^{(1)}, \quad  \text{for }~\gamma_E^{(*)} \in \widehat{\I}_{E_*},~\gamma_R^{(*)} \in \widehat{\I}_{R_*}
\end{equation*}
\end{Theorem}
\begin{proof}
	Assume that $\lambda \not \in \widehat{\mathcal{I}}_{E_0}$.
	Then from the first row of \eqref{Eq31.0} we obtain
\begin{equation}\label{Eq65}
	x = (\lambda I -E_0)^{-1} R_1^\top y.
\end{equation}
	Inserting (\ref{Eq65}) into the second row of \eqref{Eq31.0} yields
\begin{equation*}
	\underbrace{\left( R_1 (\lambda I -  E_0)^{-1} R_1^\top - \lambda I - E_1 \right)}_{\text{\normalsize$Y(\lambda) $}} y = 0.
\end{equation*}
	Applying Lemma \ref{Le1} to $Z(\lambda) = \lambda I - E_0$ and setting $u = R_1^\top y$ yields
	\begin{align}
		\nonumber 0 ={}& 	\frac{y^\top Y(\lambda) y}{y^\top y} =
		\frac{y^\top R_1 (\lambda I -  E_0)^{-1} R_1^\top y}{y^\top y} - \lambda  - \frac{y^\top E_1 y}
		{y^\top y}  \\
		    \label{yZy}  ={}& \frac{u^\top (\lambda I -  E_0)^{-1} u}{u^\top u}\frac{y^\top R_1 R_1^\top y}{y^\top y} - \lambda - \gamma_E^{(1)} = \frac{\gamma_R^{(1)}}{\lambda - \gamma_E^{(0)}} - \lambda - \gamma_E^{(1)} = \frac{p(\lambda)}{\gamma_E^{(0)}-\lambda},
    \end{align}
	applying the shorthand $p(\lambda)$ for $p(\lambda; \boldsymbol \gamma_E, \boldsymbol \gamma_R)$, i.e.,
		\[p(\lambda) = \lambda^2 - (\gamma_E^{(0)} - \gamma_E^{(1)}) \lambda- \gamma_R^{(1)} -\gamma_E^{(0)} \gamma_E^{(1)}.\]
\end{proof}
We can finally bound the roots of $p(\lambda)$:
\begin{Lemma}
	\label{Lem:optp}
	Denote $\lambda_-(\gamma_E^{(0)}, \gamma_E^{(1)}, \gamma_R^{(1)})$ and $\lambda_+(\gamma_E^{(0)}, \gamma_E^{(1)}, \gamma_R^{(1)})$ as the two roots of $p(\lambda)$, that is
	\[ \lambda_\pm = \frac{1}{2} (\gamma_E^{(0)} -\gamma_E^{(1)}) \pm \sqrt{\frac{1}{4} (\gamma_E^{(0)} +\gamma_E^{(1)})^2 +\gamma_R^{(1)}}. \]
	Then, $\lambda \in \widehat{\mathcal{I}}_- \cup \widehat{\mathcal{I}}_+, \quad$ where \\
	\begin{equation*}
		\widehat{\mathcal{I}}_- = \left[\lambda_-(\alpha_E^{(0)}, \beta_E^{(1)}, \beta_R^{(1)}) , \lambda_-(\beta_E^{(0)}, \alpha_E^{(1)}, \alpha_R^{(1)})\right],  \quad
		\widehat{\mathcal{I}}_+ = \left[ \lambda_+(\alpha_E^{(0)}, \beta_E^{(1)}, \alpha_R^{(1)}),  \lambda_+(\beta_E^{(0)}, \alpha_E^{(1)}, \beta_R^{(1)}) \right].
	\end{equation*}
\end{Lemma}
\begin{proof}
	The partial derivatives of the $p(\lambda, \gamma_*)$ have the following expressions:
  \begin{align*}
	  \frac{\partial p}{\partial \lambda}(\lambda) &= 2 \lambda - (\gamma_E^{(0)} -\gamma_E^{(1)}), &
	  \frac{\partial p}{\partial \gamma_E^{(0)}}(\lambda) &= -\lambda - \gamma_E^{(1)}, &
	  \frac{\partial p}{\partial \gamma_E^{(1)}}(\lambda) &= \lambda - \gamma_E^{(0)}, &
	  \frac{\partial p}{\partial \gamma_R^{(1)}}(\lambda) &= -1,
  \end{align*}
	We observe that if $\xi$ is a root of $p(\lambda)$, then
	\[ \xi+\gamma_E^{(1)} = \frac{\gamma_R^{(1)}}{\xi - \gamma_E^{(0)}}.\]
	It turns out that, at a root,
	\begin{equation*}
		\frac{\partial p}{\partial \gamma_E^{(0)}}(\xi) = -\frac{\gamma_R^{(1)}}{\xi - \gamma_E^{(0)}},  \qquad \quad
		\frac{\partial p}{\partial \gamma_E^{(1)}}(\xi) = \frac{\gamma_R^{(1)}}{\xi + \gamma_E^{(1)}}.
	\end{equation*}
	From the previous results, the signs of the partial derivative are summarized in the table below:
	\begin{center}
            \renewcommand{\arraystretch}{1.2}
		\begin{tabular}{|c||c|ccc|} \hline
			& $-\frac{\partial p}{\partial \lambda}(\xi)$ & $\frac{\partial p}{\partial \gamma_E^{(0)}}(\xi) $
			& $\frac{\partial p}{\partial \gamma_E^{(1)}}(\xi) $ & $\frac{\partial p}{\partial \gamma_R^{(1)}}(\xi) $ \\ \hline \hline
		$\xi = \lambda_-$ &  $+$ & $+$ & $-$ & $-$ \\
		$\xi = \lambda_+$ &  $-$ & $-$ & $+$ & $-$ \\
        \hline
	\end{tabular}
	\end{center}

         Using Lemma \ref{lem:how_roots_move} of this paper or \cite[Lem. 2.2]{BMPP_COAP24}, we obtain the thesis.
\end{proof}

\begin{Corollary}
	Any eigenvalue of \eqref{Eq31.0} satisfies
	\[ \lambda \in \widehat{\mathcal{I}}_- \cup
		\left[ \alpha_E^{(0)},  \lambda_+(\beta_E^{(0)}, \alpha_E^{(1)}, \beta_R^{(1)}) \right].
		\]
	\end{Corollary}
	\begin{proof}
		The statement follows from \Cref{Theo:p} and \Cref{Lem:optp},  and by
		observing that $p(\gamma_E^{(0)}) = -\gamma_R^{(1)} < 0$, which implies that $\widehat{\mathcal I}_{E_0} \cup \widehat{\mathcal I}_+$ = $ \left[ \alpha_E^{(0)},  \lambda_+(\beta_E^{(0)}, \alpha_E^{(1)}, \beta_R^{(1)}) \right]$.
	\end{proof}

We are now ready to characterize the eigenvalues of the preconditioned matrix
	$\widehat{\mathcal{P}}_D^{-1} \mathcal{A}$.
{\begin{Theorem}
	\label{Theo:pi}
	The eigenvalues of $\widehat{\mathcal{P}}_D^{-1} \mathcal{A}$ either belong to
	$\widehat{\mathcal I}_- \cup \widehat{\mathcal I}_{E_0} \cup \widehat{\mathcal I}_+$
	or they are solutions to the cubic polynomial equation
		\[\pi(\lambda; \boldsymbol{\gamma}_E, \boldsymbol{\gamma}_R) \equiv (\gamma_E^{(0)} -\lambda) \gamma_R^{(2)} + p(\lambda; \boldsymbol{\gamma}_E, \boldsymbol{\gamma}_R) (\lambda -\gamma_E^{(2)})  = 0.  \]
\end{Theorem}

\begin{proof}
	Assuming $\lambda \not \in \widehat{\mathcal{I}}_{E_0}$ and
	inserting (\ref{Eq65}) into \eqref{Eq33.1} yields
	 $Y(\lambda) y =  -R_2^\top z$.
	Using \Cref{Theo:p}, we have that
	 if $\lambda \not \in \widehat{\mathcal I}_- \cup \widehat{\mathcal I}_+$ then $Y(\lambda) $ is either positive or negative definite,
	 and hence invertible.
	 We can then write
	 \[ y = -Y(\lambda)^{-1} R_2^\top z,
	\]
	 whereupon substitution into (\ref{Eq34.1}) yields
	\begin{equation} \label{noE}
	\left(R_2 Y(\lambda)^{-1} R_2^\top + \lambda I  - E_2 \right) z =0.
	\end{equation}
	Let us now pre-multiply~\eqref{noE} by  $\frac{z^\top}{z^\top z} $ to establish
	\[\frac{z^\top R_2 Y(\lambda)^{-1} R_2^\top z}{z^\top z} +  \lambda -
	\frac{z^\top E_2 z}{z^\top z} =0.
	\]
	Setting $s = R_2^\top z$, and multiplying the numerator and denominator of the first term by $s^\top s$, we obtain
	\begin{equation*}
		\frac{s^\top Y(\lambda)^{-1} s}{s^\top s} \frac{z^\top R_2 R_2^\top z}{z^\top z} +\lambda - \gamma_E^{(2)} = 0.
	\end{equation*}
	Using (\ref{yZy}) and applying Lemma \ref{Le1} to $Y(\lambda)$, we have that
	\[ \frac {\gamma_E^{(0)} - \lambda}{p(\lambda)} \gamma_R^{(2)} + \lambda -\gamma_E^{(2)} = 0,\]
	the zeros of which outside $ \widehat{\mathcal I}_- \cup  \widehat{\mathcal I}_+$ characterize the eigenvalues of the preconditioned matrix,
	or equivalently the zeros of
	\[\pi(\lambda) = p(\lambda) (\lambda - \gamma_E^{(2)}) + \gamma_R^{(2)} (\gamma_E^{(0)} - \lambda), \]
    where we apply the shorthand $\pi(\lambda)$ for $\pi(\lambda; \boldsymbol{\gamma}_E, \boldsymbol{\gamma}_R)$.
\end{proof}
{The polynomial $\pi(\lambda)$ has exactly three real roots $ \mu_a$, $\mu_b$, and $\mu_c$, satisfying
\[ \mu_a < \lambda_- < 0 < \mu_b < \lambda_+ < \mu_c,\]
due to
\begin{equation*}
  \begin{aligned}
        \lim _{\lambda \to -\infty} \pi(\lambda) & = -\infty, &
        \pi(\lambda_-) & = \gamma_R^{(2)} (\gamma_E^{(0)} - \lambda_-)  > 0,  \\
          \lim _{\lambda \to +\infty} \pi(\lambda) & = +\infty, &
          \pi(0) & = (\gamma_R^{(1)} +\gamma_E^{(0)} \gamma_E^{(1)}) \gamma_E^{(2)} + \gamma_R^{(2)} \gamma_E^{(0)} > 0,   &
          \pi(\lambda_+) & = \gamma_R^{(2)} (\gamma_E^{(0)} - \lambda_+) < 0,
  \end{aligned}
\end{equation*}
where the last inequality is a consequence of the already observed inequality $p(\gamma_E^{(0)}) = -\gamma_R^{(1)} < 0$.
 It further holds that $\pi'(\mu_a) > 0$, $\pi'(\mu_b) < 0$, and $\pi'(\mu_c) > 0$.}


To bound the roots of $\pi$ we need to determine the sign of the partial derivatives of the Lagrangian with respect to its variables.
{
\begin{Lemma}
	\label{Lemma:pi}
	A root $\xi$ of $\pi$ satisfies
	\begin{equation*}
	 \xi \in \left[\mu_-^{\tLB}, \mu_-^{\tUB}\right] \cup
	\left[\mu_+^{\tLB}, \mu_+^{\tUB}\right], \end{equation*}
	 where
	 \begin{align*}
		 \mu_-^{\tLB}& = \mu_a(\alpha_E^{(0)}, \beta_E^{(1)}, \beta_R^{(1)}, \alpha_E^{(2)}, \beta_R^{(2)}), \\
		 \mu_-^{\tUB}& = \mu_a(\beta_E^{(0)}, \alpha_E^{(1)}, \alpha_R^{(1)}, \beta_E^{(2)}, \alpha_R^{(2)}), \\
		 \mu_+^{\tLB}& = \min\left\{\alpha_E^{(0)}, \mu_b(\alpha_E^{(0)}, \beta_E^{(1)}, \beta_R^{(1)}, \alpha_E^{(2)}, \alpha_R^{(2)})\right\}, \\
		 \mu_+^{\tUB}& = \mu_c(\beta_E^{(0)}, \alpha_E^{(1)}, \beta_R^{(1)}, \beta_E^{(2)}, \beta_R^{(2)}).
	 \end{align*}
\end{Lemma}
}
\begin{proof}
As before, $\xi$ denotes a generic root of $\pi$.
	\begin{align*}
		\frac{\partial \pi}{\partial \gamma_E^{(0)}}(\xi) &= \gamma_R^{(2)} -(\xi - \gamma_E^{(2)})(\xi + \gamma_E^{(1)}), &
		\frac{\partial \pi}{\partial \gamma_R^{(1)}}(\xi) &= \gamma_E^{(2)} -\xi, \\
		\frac{\partial \pi}{\partial \gamma_E^{(1)}}(\xi) &= (\xi - \gamma_E^{(0)})(\xi - \gamma_E^{(2)}), &
		\frac{\partial \pi}{\partial \gamma_R^{(2)}}(\xi) &= \gamma_E^{(0)} -\xi, \\
		\frac{\partial \pi}{\partial \gamma_E^{(2)}}(\xi) &= -p(\xi).
	\end{align*}
	At first sight, the signs of the derivatives depend on the value of some indicators and the position of $\xi$.
	However, noticing that $\pi(\xi) = 0$ implies
	\[ \xi - \gamma_E^{(2)} = \frac{\gamma_R^{(2)}(\xi - \gamma_E^{(0)})}{p(\xi)}, \]
	we have, after some algebra, also that
	\begin{align*}
		\frac{\partial \pi}{\partial \gamma_E^{(0)}}(\xi) &= -\frac{\gamma_R^{(1)} \gamma_R^{(2)} }{p(\xi)}, \\
		\frac{\partial \pi}{\partial \gamma_R^{(1)}}(\xi) &= \gamma_R^{(2)} \frac{\gamma_E^{(0)} - \xi}{p(\xi)}, \\
		\frac{\partial \pi}{\partial \gamma_E^{(1)}}(\xi) &= \gamma_R^{(2)}\frac{(\xi - \gamma_E^{(0)})^2}{p(\xi)}.
	\end{align*}
	 If $\gamma_E^{(0)} \le \mu_b$ then we set  $\mu_+^{\text{LB}} \equiv \alpha_E^{(0)}$. 
    Otherwise, we use
the fact that $\mu_a < \lambda_- < \mu_b < \lambda_+ < \mu_c$ to  devise the sign of $p(\xi)$ and that
	\[ p(\mu_a) >0, \ p(\mu_b) < 0, \ p(\mu_c) > 0, \qquad 0 < {\mu_b < \gamma_E^{(0)}} < \mu_c.\]
	Taking all the previous into account,
	the following sketch gives the desired monotonicity analysis:
	\begin{center}
            \renewcommand{\arraystretch}{1.2}
		\begin{tabular}{|c||c|ccccc|} \hline 
			& $-\frac{\partial \pi}{\partial \lambda}(\xi)$ &  $\frac{\partial \pi}{\partial \gamma_E^{(0)} }(\xi)$
			& $\frac{\partial \pi}{\partial \gamma_E^{(1)} }(\xi)$ & $\frac{\partial \pi}{\partial \gamma_R^{(1)} }(\xi)$
			& $\frac{\partial \pi}{\partial \gamma_E^{(2)} }(\xi)$ & $\frac{\partial \pi}{\partial \gamma_R^{(2)} }(\xi)$  \\ \hline \hline
		$\xi = \mu_a$ & $-$ & $-$ & $+$ & $+$ & $-$ & $+$  \\
		$\xi = \mu_b$ & $+$ & $+$ & $-$ & $-$  & $+$ & $+$  \\
		$\xi = \mu_c$ & $-$ & $-$ & $+$ & $-$  & $-$ & $-$  \\
        \hline
        \end{tabular}
	\end{center}
	 Using Lemma \ref{lem:how_roots_move} of this paper or \cite[Lem. 2.2]{BMPP_COAP24}, we obtain the thesis.
	 \end{proof}
	 \begin{Corollary}
		 \label{corInexactDouble}
		    Any  eigenvalue $\lambda$ of $\widehat{\mathcal{P}}_D^{-1} \mathcal{A}$ satisfies
	\[ \lambda \in \left[\mu_-^{\tLB}, \lambda_-(\beta_E^{(0)}, \alpha_E^{(1)}, \alpha_R^{(1)})\right]  \cup
	 \left[\mu_+^{\tLB}, \mu_+^{\tUB}\right], \]
	 \end{Corollary}
	 \begin{proof}
		 We recall that
		 $\pi(\lambda_-) > 0$, which implies that $\lambda_- > \mu_a$.
		 The statement then follows from \Cref{Theo:pi} and \Cref{Lemma:pi}.
	 \end{proof}
	 {
	 \begin{Remark}
		 Some of the bounds in Corollary \ref{corInexactDouble} refine those in \cite[Thm. 2.2]{Bradley}. In particular, it turns
		 out that $\mu_-^{\tLB}$ and $\mu_+^{\tUB}$ are  exactly as in \cite[Eq. (2.3)]{Bradley}, while
		 \begin{align*}  \lambda_-(\beta_E^{(0)}, \alpha_E^{(1)}, \alpha_R^{(1)}) &\le \lambda_-(\beta_E^{(0)}, 0, \alpha_R^{(1)})
		 = \frac{\beta_E^{(0)}-\sqrt{{\beta_E^{(0)}}^2 + 4 \alpha_R^{(1)}}}{2}, \\
                  \mu_+^{\tLB}& = \min\left\{\alpha_E^{(0)}, \mu_b(\alpha_E^{(0)}, \beta_E^{(1)}, \beta_R^{(1)}, \alpha_E^{(2)}, \alpha_R^{(2)})\right\} \\
			 &\ge  \min\left\{\alpha_E^{(0)}, \mu_b(\alpha_E^{(0)}, \beta_E^{(1)}, \beta_R^{(1)}, 0, \alpha_R^{(2)})\right\} = \mu_b(\alpha_E^{(0)}, \beta_E^{(1)}, \beta_R^{(1)}, 0, \alpha_R^{(2)})
			 =: p^+_{\min},
		 \end{align*}
         where $p^+_{\min}$ is the notation of \cite{Bradley} for the lower bound of the positive eigenvalues. It turns out that $\mu_+^{\tLB}$ can be significantly larger than $p^+_{\min}$. For instance, considering the parameter selection $\alpha_E^{(0)} = 0.01$, $\beta_E^{(1)} = 0.1$, $\beta_R^{(1)} = 2$, $\alpha_E^{(2)} = 0.1$, $\alpha_R^{(2)} = 10^{-3}$, 
		 one would obtain $\mu_+^{\tLB} = 0.01$,  whereas $p^+_{\min} \approx 5 \times 10^{-6}$.
	 \end{Remark}
	 }

\subsection{Numerical experiments: Randomly-generated matrices}

We now undertake numerical tests to validate the theoretical results of Section \ref{sec:DSPAnalysis}. For the first test, we ascertain the extremal eigenvalues of $\widehat{\mathcal{P}}_D^{-1} \mathcal{A}$, as compared to the bounds developed in Section \ref{sec:DSPAnalysis}, on randomly-generated linear systems. Specifically, we run $3^6 = 729$ different synthetic test cases in the simplified case with $E_1, E_2 \equiv 0$, combining the values of the extremal eigenvalues of the symmetric positive definite matrices involved in the previous discussion, reported in \Cref{TabPar}. Each test case has been run 25 times, generating random matrices that satisfy the relevant spectral properties, and we record the most extreme eigenvalues for each test case.

\begin{table}[h!]
\begin{center}
\begin{tabular}{|c||rrl|}
\hline
$\alpha_E^{(0)}$, $\alpha_R^{(1)}$, $\alpha_R^{(2)}$ & 0.1 & 0.3 & 0.9 \\
$\beta_E^{(0)}$, $\beta_R^{(1)}$, $\beta_R^{(2)}$ & 1.2 & 1.8 & 5 \\
\hline
\end{tabular}
\end{center}
\caption{Extremal eigenvalues of the relevant symmetric positive definite matrices used in the verification of the bounds.}
\label{TabPar}
\end{table}

\begin{figure}[h!]
\begin{center}
\hspace{-5mm}
\includegraphics[width=7cm]{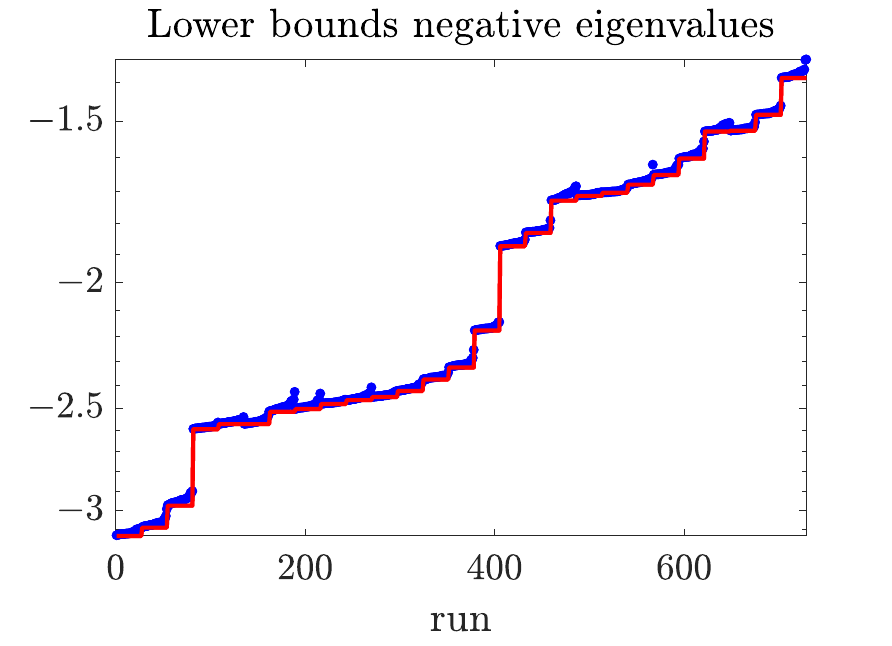}
\hspace{-5mm}
\includegraphics[width=7cm]{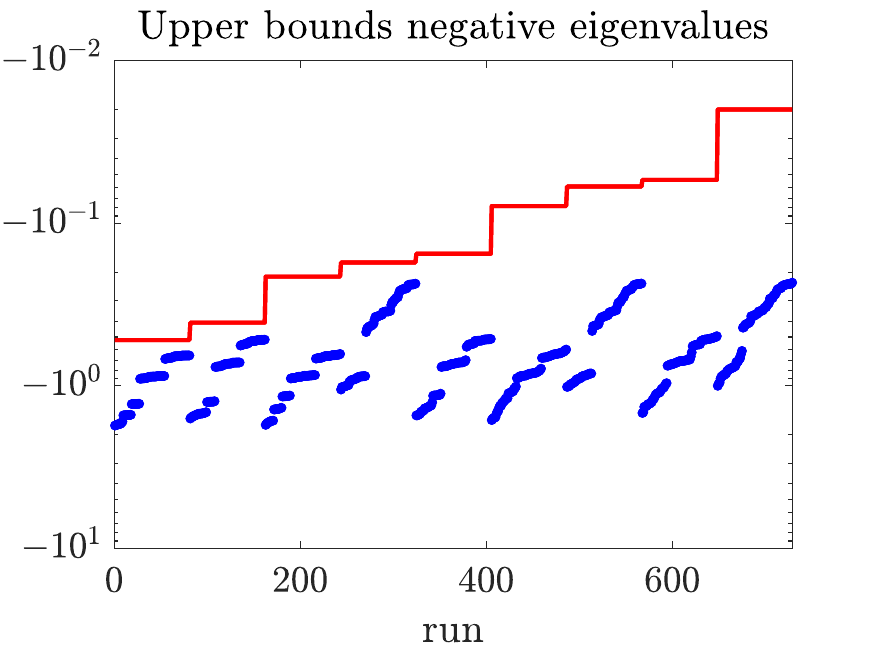}
\hspace{-5mm}
\includegraphics[width=7cm]{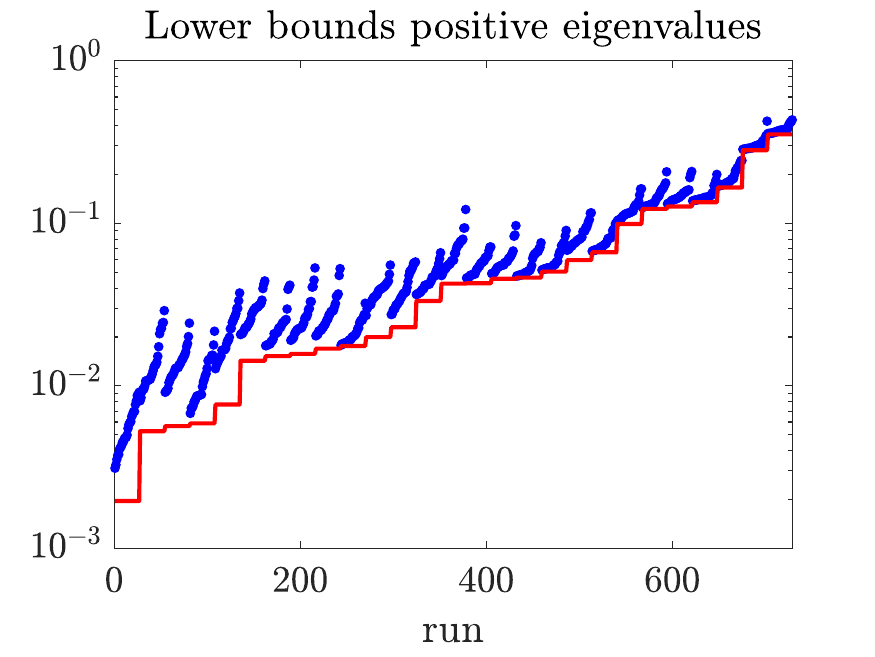}
\hspace{-5mm}
\includegraphics[width=7cm]{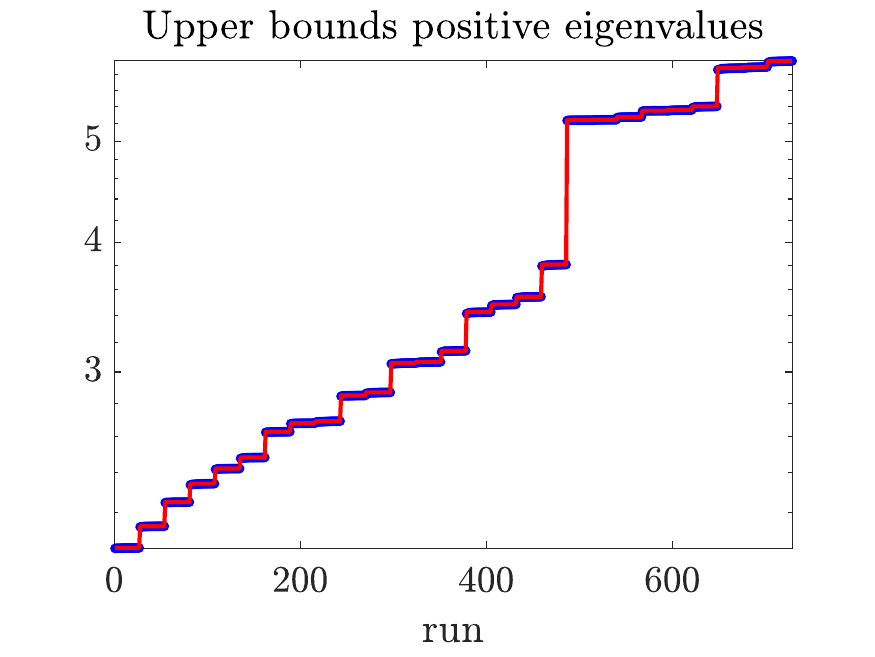}
\end{center}
\caption{Extremal eigenvalues of the preconditioned matrix (blue dots) and bounds (red line) after 25 runs with each combination of the parameters from \Cref{TabPar}.}
\label{eigvsbounds}
\end{figure}

In more detail, the dimensions $n_0$, $n_1$, and $n_2$ are computed using the closest integer to \texttt{50+10*rand}, using {\scshape Matlab}'s \texttt{rand} function, re-computing as necessary to ensure that $n_0 \geq n_1 \geq n_2$. The matrices $A_k$ and $B_k$ are computed using {\scshape Matlab}'s \texttt{randn} function, whereupon we take the symmetric part of $A_k$ and then add an identity matrix multiplied by the absolute value of the smallest eigenvalue (multiplied by 1.01 if $k=0$), to ensure symmetric positive semi-definiteness (definiteness if $k=0$). We then choose $\widehat{S}_0$ as a linear combination of $A$ and the identity matrix, such that the eigenvalues of $E_0$ are contained in $[\alpha_E^{(0)}, \beta_E^{(0)}]$, and similarly to construct $\widehat{S}_1$ and $\widehat{S}_2$. In \Cref{eigvsbounds} we present the (ordered) computed extremal eigenvalues and theoretical bounds. We notice that the plots indicate (for these problems) that three out of four bounds capture the behaviour of the eigenvalues very well, while only the upper bounds on the negative eigenvalues are not as tight. 

\subsection{Numerical experiments: PDE-constrained optimization}

\begin{table}
\centering
\caption{Computed eigenvalues of $\widehat{\mathcal{P}}_D^{-1} \mathcal{A}$ (denoted `\text{Comp}') and theoretical bounds (denoted `\text{Bound}'), for PDE-constrained optimization problem with $h = 2^{-6}$, $\beta \in \{ 1, 10^{-3} \}$, and a range of Chebyshev semi-iterations \texttt{Cheb}.}\label{ResultsTable2}
\begin{tabular}{|c||c|c||c|c||c|c||c|c|}
\hline
 & \multicolumn{8}{c|}{$\beta = 1$} \\ \cline{2-9}
\texttt{Cheb} & $\text{Bound}_l^-$ & $\text{Comp}_l^-$ & $\text{Comp}_u^-$ & $\text{Bound}_u^-$ & $\text{Bound}_l^+$ & $\text{Comp}_l^+$ & $\text{Comp}_u^+$ & $\text{Bound}_u^+$ \\ \hline \hline
1 & $-1.9288$ & $-1.7075$ & $-0.7087$ & $-0.0944$ & 0.0537 & 0.2688 & 4.4527 & 4.4527 \\ \hline
3 & $-1.3239$ & $-1.3065$ & $-0.8286$ & $-0.5335$ & 0.3751 & 0.4294 & 4.2769 & 4.2769 \\ \hline
5 & $-1.2554$ & $-1.2537$ & $-0.7795$ & $-0.6084$ & 0.4370 & 0.4434 & 4.2582 & 4.2582 \\ \hline
10 & $-1.2470$ & $-1.2470$ & $-0.7733$ & $-0.6180$ & 0.4450 & 0.4450 & 4.2559 & 4.2559 \\ \hline
\end{tabular}

\begin{tabular}{|c||c|c||c|c||c|c||c|c|}
\hline
 & \multicolumn{8}{c|}{$\beta = 10^{-3}$} \\ \cline{2-9}
\texttt{Cheb} & $\text{Bound}_l^-$ & $\text{Comp}_l^-$ & $\text{Comp}_u^-$ & $\text{Bound}_u^-$ & $\text{Bound}_l^+$ & $\text{Comp}_l^+$ & $\text{Comp}_u^+$ & $\text{Bound}_u^+$ \\ \hline \hline
1 & $-1.9288$ & $-1.7070$ & $-0.7087$ & $-0.0944$ & 0.0537 & 0.2688 & 4002.7 & 4002.7 \\ \hline
3 & $-1.3239$ & $-1.3065$ & $-0.6639$ & $-0.5335$ & 0.3751 & 0.4294 & 4002.7 & 4002.7 \\ \hline
5 & $-1.2554$ & $-1.2537$ & $-0.6233$ & $-0.6084$ & 0.4370 & 0.4434 & 4002.7 & 4002.7 \\ \hline
10 & $-1.2470$ & $-1.2470$ & $-0.6182$ & $-0.6180$ & 0.4450 & 0.4450 & 4002.7 & 4002.7 \\ \hline
\end{tabular}
\end{table}

\begin{table}
\centering
\caption{Number of {\scshape Minres} iterations required for solution of PDE-constrained optimization problem, for a range of Chebyshev semi-iterations \texttt{Cheb}.}\label{ResultsTable3}
\begin{tabular}{|c||c|c|c|c|}
\hline
$h$ & $2^{-5}$ & $2^{-5}$ & $2^{-6}$ & $2^{-6}$ \\ \hline
$\texttt{Cheb} \backslash \beta$ & $~1~$ & $~10^{-3}~$ & $~1~$ & $~10^{-3}~$ \\ \hline \hline
1 & 104 & 135 & 104 & 138 \\ \hline
3 & 31 & 75 & 31 & 74 \\ \hline
5 & 21 & 72 & 21 & 68 \\ \hline
10 & 16 & 64 & 15 & 58 \\ \hline
\end{tabular}
\end{table}

To test our theoretical results of Section \ref{sec:DSPAnalysis} on a more practical PDE-based setting, we consider a PDE-constrained optimization example, and present the theoretical and computed, extremal (negative and positive) eigenvalues of $\widehat{\mathcal{P}}_D^{-1} \mathcal{A}$. Following literature such as \cite[Sec. 5]{BMPP_COAP24}, \cite{MNN}, and \cite[Sec. 4.1]{pearson2023symmetric}, we examine a problem of the form
\begin{equation*}
\min_{y,u} \quad \frac{1}{2}\left\|y-\widehat{y}\right\|_{L^2(\partial\Omega)}^2+\frac{\beta}{2}\left\|u\right\|_{L^2(\Omega)}^2 \qquad \text{s.t.} \quad \left\{\begin{array}{rl}
-\Delta y+y+u=0 & \text{in }\Omega, \\
\frac{\partial y}{\partial n}=0 & \text{on }\partial\Omega, \\
\end{array}\right.
\end{equation*}
where $\Omega := (0,1)^2$ with boundary $\partial\Omega$, $y$ and $u$ are \emph{state} and \emph{control variables}, $\widehat{y}$ is a specified \emph{desired state}, and $\beta>0$ is a regularization parameter. For these tests, the desired state is obtained from solving the forward problem with `true' control $4x_1(1-x_1)+x_2$, with $x_1$ and $x_2$ the spatial coordinates. We discretize these problems using P1 finite elements, to obtain a linear system
\begin{equation}\label{PDECO_system}
\left[\begin{array}{ccc}
\beta M & M & 0 \\ M & 0 & L \\ 0 & L & M_{\partial\Omega} \\
\end{array}\right]\left[\begin{array}{c}
u_h \\ p_h \\ y_h \\
\end{array}\right]=\left[\begin{array}{c}
0 \\ 0 \\ \widehat{y}_h \\
\end{array}\right],
\end{equation}
with $y_h$, $u_h$, $p_h$, and $\widehat{y}_h$ the discretized state, control, \emph{adjoint}, and desired state. The matrix $M$ is a finite element mass matrix, $L$ is the sum of finite element stiffness and mass matrices, and $M_{\partial\Omega}$ a boundary mass matrix.

Consistent with the structure of $\mathcal{P}_D$ above, we precondition the system \eqref{PDECO_system} with
\begin{equation*}
\widehat{\mathcal{P}}_D = \left[\begin{array}{ccc}
\beta \widehat{M} & 0 & 0 \\ 0 & \frac{1}{\beta} \widehat{M} & 0 \\ 0 & 0 & \beta L M^{-1} L \\
\end{array}\right],
\end{equation*}
where $\widehat{M}$ denotes a number of iterations of Chebyshev semi-iteration applied to $M$ (see \cite{GVI,GVII,WathenRees}) with Jacobi splitting. We vary the number of inner iterations, denoted \texttt{Cheb}, in our experiments. For the purposes of these results, we apply the inverse of $S_2$ ``exactly'' using a direct solver, but could easily apply a multigrid method (or similar) instead. As in \cite[Sec. 5]{BMPP_COAP24}, our theoretical bounds may be obtained using known results:
\begin{equation*}
[\alpha_E^{(0)},\beta_E^{(0)}] \in [1-\omega, 1+\omega], \qquad \text{where }~\omega = 1\left/T_{\texttt{Cheb}} \left( \frac{\lambda_{\max}^M+\lambda_{\min}^M}{\lambda_{\max}^M-\lambda_{\min}^M} \right)\right. \equiv 1\left/T_{\texttt{Cheb}} \left( \frac{5}{3} \right),\right.
\end{equation*}
with $T_{\texttt{Cheb}}$ the $\texttt{Cheb}$-th Chebyshev polynomial, and $\lambda_{\min}^M$ and $\lambda_{\max}^M$ the minimum and maximum eigenvalues of the mass matrix preconditioned by its diagonal, in this case $\frac{1}{2}$ and $2$ respectively \cite{Wathen87}. From this, we may take the bounds $\alpha_R^{(1)} = (\alpha_E^{(0)})^2$, $\beta_R^{(1)} = (\beta_E^{(0)})^2$, $\alpha_R^{(2)} = \alpha_E^{(0)}$, and $\beta_R^{(2)} = \beta_E^{(0)}$. Further, it holds that $\alpha_E^{(1)} = \beta_E^{(1)} = \alpha_E^{(2)} \equiv 0$, and we compute the value of $\beta_E^{(2)}$ directly.

In Table \ref{ResultsTable2} we show the extremal computed negative and positive eigenvalues ($\text{Comp}_l^-$, $\text{Comp}_u^-$, $\text{Comp}_l^+$, $\text{Comp}_u^+$) of $\widehat{\mathcal{P}}_D^{-1} \mathcal{A}$ with mesh parameter $h = 2^{-6}$ and $\beta \in \{ 1, 10^{-3} \}$, with different numbers of Chebyshev semi-iterations applied to approximate $M^{-1}$. We also provide the analytic bounds ($\text{Bound}_l^-$, $\text{Bound}_u^-$, $\text{Bound}_l^+$, $\text{Bound}_u^+$), obtained using the above theory. We also ran the analogous results for $h = 2^{-5}$, but do not present these as they are very similar to those for $h = 2^{-6}$. In Table \ref{ResultsTable3} we present the {\scshape Minres} \cite{minres} iteration numbers required to solve the systems from all problem settings, with $h = 2^{-5}$ and $h = 2^{-6}$, to a relative tolerance of $10^{-10}$. The dimensions of the systems for $h = 2^{-5}$ and $h = 2^{-6}$ are 3267 and 12675, respectively.

Table \ref{ResultsTable2} demonstrates that the theoretical bounds of this section are highly illustrative for a practical, PDE-based example. As also observed in Figure \ref{eigvsbounds}, the weakest bound is that for the negative eigenvalue of smallest magnitude, but as long as $\texttt{Cheb}$ is larger than $1$ this is still of the correct order of magnitude. As expected, as the quality of the approximation of $M$ improves through increasing \texttt{Cheb}, the more precise are the eigenvalue bounds. This also manifests itself in a reduction of iteration numbers in Table \ref{ResultsTable3}. The iteration numbers of Table \ref{ResultsTable3} also reflect the eigenvalue bounds of Table \ref{ResultsTable2} in that these increase for smaller $\beta$ due to $\text{Comp}_u^+$ being significantly larger, and they are robust (indeed, almost independent) with respect to $h$.

\section{Conclusion} \label{sec:conc}

We analyzed spectral bounds for the application of block diagonal Schur complement preconditioners to symmetric multiple saddle-point systems. The analysis is based on recursively defined polynomials, the zeros of which coincide with the eigenvalues of interest. We derived bounds on the positive and negative zeros, showing that the extremal eigenvalues are bounded from above by two in modulus independently of $N$, whereas the eigenvalues closest to zero converge to zero for $N \to \infty$. We provided a general perturbation result for the approximate application of the preconditioners based on a backward analysis argument and practical bounds based on field-of-value indicators for double saddle-point systems. Numerical experiments complement our theoretical findings and demonstrate that the bounds are in close agreement with the numerically-determined eigenvalues.


\section*{Acknowledgements}
We would like to thank Michele Bergamaschi for key discussions on the proof of Theorem \ref{Theo1}.
AM acknowledges support of the INdAM-GNCS Project CUP\underline{ }E53C23001670001. The work of AM was carried out within the PNRR research activities of the consortium iNEST (Interconnected Nord-Est Innovation Ecosystem), funded by the European Union Next-GenerationEU
PNRR -- Missione 4 Componente 2, Investimento 1.5 -- D.D. 1058 23/06/2022, ECS$\_$00000043).
This manuscript reflects only the Authors' views and opinions, neither the European Union nor the European Commission can be considered responsible for them.
JWP acknowledges support of the Engineering and Physical Sciences (EPSRC) UK grant EP/S027785/1. 
AP acknowledges support of the German Federal Ministry of Education and Research under grant 05M22VHA.

\bibliographystyle{siam}
\bibliography{BDP}
\end{document}